\documentclass[a4paper,11pt]{amsart}
\usepackage{amssymb,amsmath,amsthm}
\numberwithin{equation}{section}
\usepackage{paralist}               
\usepackage[margin=1.0in]{geometry}
\usepackage[titletoc]{appendix}     
\usepackage{mathrsfs}
\usepackage{setspace}
\usepackage[final]{graphicx}
\usepackage{epstopdf}
\usepackage{subcaption}
\usepackage{float}
\usepackage[bookmarks,bookmarksnumbered,bookmarksopen,breaklinks,linktocpage]{hyperref}
\newcommand{\ds}{\displaystyle}

\newcommand{\dd}{\mathrm{d}}

\newtheorem{Theorem}{Theorem}[section]
\newtheorem{Proposition}{Proposition}[section]

\theoremstyle{definition}           
\newtheorem{defn}{Definition}[section]
\theoremstyle{remark}

\usepackage{multirow}
\usepackage{color,soul}
\usepackage{soul}
\usepackage{array}
\newcolumntype{L}[1]{>{\raggedright\let\newline\\\arraybackslash\hspace{0pt}}m{#1}}
\newcolumntype{C}[1]{>{\centering\let\newline\\\arraybackslash\hspace{0pt}}m{#1}}
\newcolumntype{R}[1]{>{\raggedleft\let\newline\\\arraybackslash\hspace{0pt}}m{#1}}

\begin{document}
	\title[A fully nonlinear coagulation-fragmentation model]{Numerical schemes for a fully nonlinear coagulation-fragmentation model coming from wave kinetic theory}
	
	\author[Arijit Das]{Arijit Das}
	
	\address[Arijit Das]{Department of Mathematics \\
		Chair for Dynamics, Control, Machine Learning and Numerics (Alexander von Humboldt Professorship), Friedrich-Alexander University Erlangen-Nuremberg, Erlangen, Germany.}
	
	\email[Arijit Das]{\href{mailto:arijit.das@fau.de}{arijit.das@fau.de}}
	
	\author[Minh-Binh Tran]{Minh-Binh Tran}
	\address[Minh-Binh Tran]{Department of Mathematics \\
		Texas A\&M University, College Station, TX, 77843 USA}
	
	\email[Minh-Binh Tran]{\href{mailto:minhbinh@tamu.edu}{minhbinh@tamu.edu}}
	
	%
	%

	\begin{abstract}
		This article introduces a novel numerical approach, based on  Finite Volume Techniques, for studying fully nonlinear coagulation-fragmentation models, where both the coagulation and fragmentation components of the collision operator are nonlinear. The models come from  $3-$wave kinetic equations, a pivotal framework in wave turbulence theory. Despite the importance of wave turbulence theory in physics and mechanics, there have been very few numerical schemes for  $3-$wave kinetic equations, in which no  additional assumptions are manually imposed on the evolution of the solutions, and the current manuscript provides one of the first of such schemes. To the best of our knowledge, this also is the first numerical scheme capable of accurately capturing the long-term asymptotic behavior of solutions to a fully nonlinear coagulation-fragmentation model. The scheme is implemented on some test problems, demonstrating strong alignment with theoretical predictions of energy cascade rates, rigorously obtained in the work \cite{soffer2020energy}. We further introduce a weighted Finite Volume variant to ensure energy conservation across varying degrees of kernel homogeneity. Convergence and first-order consistency are established through theoretical analysis and verified by experimental convergence orders in test cases.
	\end{abstract}
	\maketitle
	
\section{Introduction}\label{Sec_1}
Over the past six decades, the theory of  wave turbulence has been shown to play  important roles in a vast range of physical and mechanical examples including inertial waves due to rotation, Alfv\'en wave turbulence in the solar wind, waves in plasmas of fusion devices, and many others. Building on Peierls' foundational work \cite{peierls1929kinetischen}, the theory's modern development has been driven by contributions from Benney and Saffman \cite{benney1966nonlinear}, Zakharov and {Filonenko} \cite{zakharov1967weak},  Benney and Newell \cite{benney1969random} and especially the work of Hasselmann \cite{hasselmann1962non,hasselmann1974spectral}. These efforts culminated in the formulation of $3-$wave and $4-$wave kinetic equations, which describe the energy distribution among weakly interacting waves. The general form of 3-wave kinetic equations reads
\begin{equation}\label{WT1}
	\begin{aligned}
		\partial_tf(t,p) =&   \iint_{\mathbb{R}^{2d}} \Big[ R_{p,p_1,p_2}[f] - R_{p_1,p,p_2}[f] - R_{p_2,p,p_1}[f] \Big] \dd p_1\dd p_2,  \ 
		f(0,p) \ =  \ f_0(p),
	\end{aligned}
\end{equation}
where  $f(t,p)$ is the  wave density  at  wavenumber $p\in \mathbb{R}^d$, $d \ge 2$ and $f_0(p)$ is the initial condition.  Moreover,
\begin{equation}\label{WT2}	R_{p,p_1,p_2} [f]:=  |V_{p,p_1,p_2}|^2\delta(p-p_1-p_2)\delta(\omega -\omega_{1}-\omega_{2})(f_1f_2-ff_1-ff_2) , \end{equation}
with the short-hand notations $f = f(t,p)$, $\omega = \omega(p)$ and $f_j = f(t,p_j),$ $\omega_j = \omega(p_j)$, for  $p$, $p_j$, $j\in\{1,2\}$. The quantity $\omega(p)$ is the dispersion relation of the waves. For a deeper physical understanding, we refer to the comprehensive works in \cite{nazarenko2011wave, pomeau2019statistical, zakharov2012kolmogorov}.

In a related context,  coagulation and fragmentation kinetic equations arise in several physical  and mechanical contexts: the formation of aerosols,  polymers, celestial bodies on astronomical scales or colloidal aggregates   (see \cite{banasiak2019analytic,canizo2010regularity,cristian2023long,cristian2024fast,degond2017coagulation,escobedo2003gelation,liu2019high,jang2024discrete,menon2008scaling,perthame2005exponential,stewart1989global,tran2022coagulation} and the references therein).  Coagulation-fragmentation kinetic equations describe the behavior of particles that coagulate or fragment due to the mutual interactions or some external force. The primary quantity of interest in this field is the particle number density function, typically denoted $n_\omega(t)$, representing the density of clusters of size  $\omega$ at time $t$. For pure coagulation, the Smoluchowski coagulation equation serves as the classical model of particulate processes.   Despite the difference in focus, there is a conceptual analogy between coagulation-fragmentation kinetics and $3-$wave kinetic models: the transfer of energy between scales in wave turbulence is analogous to the transfer of mass between clusters in coagulation processes. This observation offers a quantitative bridge between the two fields. Building on this analogy, Connaughton and Newell proposed an approximation of the $3-$wave kinetic equations \eqref{WT1}-\eqref{WT2} using a nonlinear coagulation-fragmentation model \cite{connaughton2010dynamical}. In contrast to standard models, Connaughton-Newell’s formulation incorporates nonlinearities in both the coagulation and fragmentation terms of the collision operator.  Connaughton-Newell's model is given in \cite{connaughton2010dynamical}

\begin{align}\label{1.1}
	\begin{split}
		&\frac{\partial N_{\omega}}{\partial t} = \mathcal{Q}\left[N_\omega\right](t), \quad \omega \in \mathbb{R}_+, \quad N_\omega(0)= N_\omega^{in}.
	\end{split}
\end{align}
The operator $\mathcal{Q}\left[N_\omega\right](t)$ can  be expressed as
\begin{align}\label{eq_1}
	\mathcal{Q}\left[N_\omega\right](t) = S_1\left[N_{\omega}\right] - S_2\left[N_{\omega}\right] - S_3\left[N_{\omega}\right],  
\end{align}
and
\begin{align*}
	\begin{split}
		&S_1\left[N_{\omega}\right] = \int_{0}^{\omega}K_1(\omega-\mu,\mu)N_{\omega-\mu}N_{\mu}\dd \mu - \int_{\omega}^\infty K_1(\mu-\omega,\omega)N_{\mu-\omega}N_{\omega}\dd \mu  - \int_{0}^{\infty} K_1(\omega,\mu) N_\omega N_{\mu}\dd \mu \\
		& \hspace{1.2cm}= \int_{0}^{\omega}K_1(\omega-\mu,\mu)N_{\omega-\mu}N_{\mu}\dd \mu - 2 \int_{0}^{\infty} K_1(\omega,\mu) N_\omega N_{\mu}\dd \mu,\\
		&S_2\left[N_{\omega}\right] = - \int_{0}^{\omega} K_2(\mu,\omega-\mu)N_\omega N_{\mu}\dd \mu + \int_{\omega}^{\infty} K_2(\omega,\mu-\omega)N_{\mu-\omega}N_{\mu}\dd \mu + \int_{0}^{\infty} K_2(\omega,\mu) N_\omega N_{\omega+\mu}\dd \mu,\\
		&  S_3\left[N_{\omega}\right] =
		-\int_{0}^{\omega}K_3(\mu,\omega-\mu)N_\omega N_{\omega-\mu}\dd \mu + \int_{\omega}^{\infty} K_3(\omega,\mu-\omega)N_{\omega}N_{\mu}\dd \mu + \int_{0}^{\infty}K_3(\omega,\mu) N_{\mu} N_{\omega+\mu}\dd \mu.
	\end{split}
\end{align*}
The wave frequency spectrum $N_\omega$ is defined such that $\ds \int_{\omega_{1}}^{\omega_{2}}N_\omega\dd\omega$ represents the total wave action in the frequency band $\ds \left[\omega_{1},\omega_{2}\right]$. To analyze solutions to the nonlinear coagulation-fragmentation model \eqref{1.1}, we define the $p-$th moment as:
\begin{align*}
	M_p(t) = \int_{0}^{\infty} \omega^p N_\omega(t)\dd \omega,\quad\text{for all}\quad p\ge 0.
\end{align*}
In this regard, the total number of wave $N$ and total wave energy $E$ can be obtained from the zeroth and first moments, respectively. Specifically, we have:
\begin{align*}
	N= \int_{0}^{\infty} N_\omega \dd \omega, \qquad E= \int_{0}^{\infty} \omega N_\omega \dd \omega.
\end{align*}
Also, the nonnegative homogeneous functions $K_i(\omega,\mu)$ $(i=1,2,3)$ represent the wave interaction kernels. Among these kernels, $K_1$ facilitates the forward transfer of energy, while $K_2$ and $K_3$ are associated with the back-scattering of energy.   

\subsection{Summary of previous numerical schemes for \eqref{1.1}}\label{Sub:IntroNumerical}

Since developing numerical schemes for 3-wave kinetic equations is highly important to understand the time evolution of the solutions, in \cite{connaughton2010dynamical},  the authors constructed  numerical schemes for \eqref{1.1}, where the kernels are chosen as
\begin{align}\label{1.2}
	K_1\left(\omega,\mu\right)=K_2\left(\omega,\mu\right)=K_3\left(\omega,\mu\right)=\left(\omega\mu\right)^{\frac{\lambda}{2}}.
\end{align} 
In the weak turbulence theory, one of the most important results  \cite{zakharov1967weak,zakharov2012kolmogorov} is  the existence of the so-called  Kolmogorov-Zakharov spectra, which is a class of { time-independent solutions}  of equation \eqref{1.1}:  $$N_\omega \approx C\omega^{-\kappa}, \ \ \ \kappa>0.$$

The goal of \cite{connaughton2010dynamical} is to show that $N_\omega$ numerically evolves to the Kolmogorov-Zakharov spectra.
To this end, the solution  $N_\omega$  is then hypothesized to evolve toward a scaling (self similar) form $
N_\omega \sim S(t)^a F\left(\frac{\omega}{S(t)}\right),
$ {as $S(t) \longrightarrow 0$ and $\omega\longrightarrow \infty$
	and $\sim$ denote the scaling limit}. For $x=\omega/S(t)$ the total energy corresponding to the self similar profile $F$ is given by
\begin{align}\label{1.4}
	\int_{0}^{\infty} \omega{ S(t)^a}  F\left(\frac{\omega}{S(t)}\right)\dd \omega=\int_{0}^{\infty} x S(t)^{a+2} F(x)\dd x = S(t)^{a+2} \int_{0}^{\infty}x F(x) \dd x \sim \mathcal{O} \left(S(t)^{a+2}\right),
\end{align} 
which grows with the rate $\ds \mathcal{O} \left(S(t)^{a+2}\right)$. Now plugging this ansatz into the equation \eqref{1.1} with the choice \eqref{1.2}, one can obtain the system
\begin{align}\label{1.5}
	\begin{split}
		&\frac{\dd S(t)}{\dd t} = S(t)^\xi, \quad \text{with}\quad \xi= \lambda+a+2,\ 
		aF(x) + x\dot{F}(x) = \mathcal{Q}\left[F\right](x).
	\end{split}
\end{align} 
The authors of \cite{connaughton2010dynamical} assumed that  the self-similar profile exhibits a power-law behavior, i.e. $\ds F(x) \sim Ax^{-m}$ as $x\sim 0$ then $m=\lambda+1$.  In   \cite{connaughton2010dynamical}, the authors consider  $\lambda \in \left[0,2\right]$ under another assumption that the total energy of the system is  not  conserved in time but grows linearly with time
$
\int_{0}^{\infty} \omega N_\omega \dd \omega =Jt
$. 
This leads to the ODE $\ds \frac{\dd S}{\dd t}= \frac{J}{\left(a+2\right)\int_{0}^{\infty}xF(x)\dd x} S(t)^{-a-1}$. By comparing this with the equation \eqref{1.5}, one obtains $\ds a= -\frac{\lambda+3}{2}$, the Kolmogorov--Zakharov-value. However, in this case the self-similar profile $F\sim x^{-\frac{\lambda+3}{2}}$ makes the integral 	\begin{align}\label{1.6}
	\int_{0}^{\infty}xF(x)\dd x,
\end{align} divergent. Consequently, further assumptions need to be imposed on the evolution of the solution $N_\omega$ itself to make sure that \eqref{1.6} converge. {\it As a conclusion, the numerical results of \cite{connaughton2010dynamical} lead   to the establishment of the Kolmogorov-Zakharov (KZ) spectra, under several assumptions manually imposed on the evolution of the solution. } 

\subsection{Summary of previous theoretical findings for \eqref{1.1}: A new phenomenon discovered in \cite{soffer2020energy}}\label{Sub:IntroTheore}

{ As discussed in Subsection \ref{Sub:IntroNumerical}, several assumptions must be manually imposed on the evolution of the solution \( N_\omega \) in order to observe the KZ-spectra. A natural question arises: \textit{Can the KZ-spectra be observed without imposing such assumptions on the evolution of \( N_\omega \)?} This question was addressed in the theoretical work \cite{soffer2020energy}, where the authors rigorously proved that, in the absence of these imposed assumptions, the solutions display a completely different behavior, which we summarize below.
} They  demonstrated that the solution to the  fully nonlinear coagulation-fragmentation model exhibits the property that the energy cascades from small wave numbers to large wave numbers. More precisely, they proved that the energy on the interval $\left[0,\infty\right)$ is a non-increasing function in time. That is for all $T_1>0$ we can always find a larger time $T_2>T_1$ such that
\begin{align}\label{1.8}
	\int_{0}^{\infty} \omega N_\omega(T_2)\dd \omega < \int_{0}^{\infty} \omega N_\omega(T_1)\dd \omega.
\end{align}
They also decomposed the energy of a solution $g_\omega(t) = \omega N_\omega(t)$ at any time $t$ as follows
\begin{align}\label{1.10}
	g_\omega(t) = \bar{g}_\omega(t) + \tilde{g}(t) \delta_{\{\omega=\infty\}}.
\end{align}
Where the nonnegative function $\bar{g}$ is the regular part and $\tilde{g}$ is the singular part which is a measure. They are both characterized by the properties $\bar{g}_\omega(0) = g_\omega(0)$ and $\tilde{g}(0)=0$. This indicates that initially, the energy is concentrated in the regular part, and as time progresses, the energy gradually accumulates at $\{\omega=\infty\}$. In this context, there exists a positive time $t_1^\ast$ known as \emph{first blow-up time} for which $\tilde{g}(t) >0$ for all $t>t_1^\ast$. Moreover, after the first blow-up time there exist \emph{infinitely many blow-up times} represented by the sequence $\ds 0< t_1^\ast<t_2^\ast<\cdot\cdot\cdot<t_n^\ast<\cdot\cdot\cdot,$ satisfying $\ds \bar{g}_\omega(t_1^\ast)> \bar{g}_\omega(t^\ast_2)>\cdot\cdot\cdot> \bar{g}_\omega(t_n^\ast)>\cdot\cdot\cdot  ~\text{and}~~	\tilde{g}(t_1^\ast)< \tilde{g}(t_2^\ast)<\cdot\cdot\cdot<\tilde{g}(t_n^\ast)<\cdot\cdot\cdot$.
If we consider $\chi_{[0,R]}(\omega)$ to be the 
{cut-off function} of $\omega$ in the finite domain $[0,R]$, the equivalent form of the multiple blow-up time phenomena  can be represented for any arbitrary truncation parameter $R$, 
\begin{align}\label{1.15}
	\int_{0}^{R}\omega N_\omega(t) \dd \omega = \int_{\mathbb{R}_+}\chi_{[0,R]}(\omega) \omega N_\omega \dd \omega \le \mathcal{O}\left(\frac{1}{\sqrt{t}}\right), \quad\text{as}\quad t \longrightarrow\infty.
\end{align}

We also note that theoretical studies of $3-$wave kinetic equations have been done in diverse contexts, such as phonon interactions in an-harmonic crystal lattices \cite{CraciunBinh,EscobedoBinh,GambaSmithBinh,tran2020reaction}, capillary waves \cite{nguyen2017quantum,walton2022deep}, beam waves \cite{rumpf2021wave}, stratified ocean flows \cite{GambaSmithBinh}, and Bose-Einstein Condensates \cite{cortes2020system,EPV,escobedo2023linearized1,escobedo2023linearized,ToanBinh,nguyen2017quantum,soffer2018dynamics}. 
\subsection{Novelty of our work: Construction of numerical schemes to numerically confirm the established theoretical findings of  \cite{soffer2020energy} discussed in Subsection \ref{Sub:IntroTheore}}\label{Sub:IntroNovel}

{The primary goal of this article is to develop a numerical scheme for the fully nonlinear coagulation-fragmentation model \eqref{1.1} that, unlike the approaches in \cite{connaughton2010dynamical} (discussed in Subsection \ref{Sub:IntroNumerical}), does not impose additional assumptions on the evolution of solutions. Because no such assumptions are manually introduced, we expect the numerical solutions to reflect the behavior rigorously established in \cite{soffer2020energy}, as outlined in Subsection \ref{Sub:IntroTheore}.} This scheme is designed for several choices of the kernels:
$$K_1(\omega,\mu)= (\omega\mu)^\theta,\ \ \ K_2(\omega,\mu)= (\omega\mu)^\gamma, \ \ \ K_3(\omega,\mu) = (\omega\mu)^\delta.$$
{\it In our numerical tests, the parameters are chosen such that $0<\theta, \gamma,\delta\le 1$, which are precisely the case considered in the  theoretical work \cite{soffer2020energy}.}

Our finite volume scheme allows for the numerical observation of the multiple blow-up time phenomenon, $\ds t_1^\ast<t_2^\ast<\cdot\cdot\cdot<t_n^\ast<\cdot\cdot\cdot$, as well as the verification of the energy decay rate estimate given by \eqref{1.15} for the fully nonlinear coagulation-fragmentation model \eqref{1.1}. 

{It is worth noting that in \cite{walton2023numerical} and \cite{walton2024numerical}, the authors proposed numerical schemes for a { specific $3-$wave kinetic equation}, derived from the theory of capillary and acoustic waves. In these works, the authors formulated a conservative representation of the kinetic model and proposed   numerical schemes demonstrating good agreement with theoretical decay phenomena. However, these schemes relied on the { exact formulation \eqref{WT1}-\eqref{WT2}} of a {  specific $3-$wave kinetic equation}, rather than the { approximation} of { all 3-wave kinetic equations} by a fully nonlinear coagulation-fragmentation model. Since Connaughton-Newell’s model represents one of the first fully nonlinear models in the literature, it is important to numerically investigate the behavior of its solutions. }

From this discussion, it is clear that developing a numerical scheme for the fully nonlinear coagulation-fragmentation model \eqref{1.1} is a pressing and unexplored task in the literature. At the same time the numerical observation of the theoretical results for the time dependent solution also becomes necessary for both scientific understanding and applications. {\it Therefore, the primary objective of this article is to develop a novel numerical scheme for the fully nonlinear coagulation-fragmentation model \eqref{1.1} to numerically verify the theoretical findings rigorously proved in \cite{soffer2020energy}.}  More specifically, the primary focus of this study is on the derivation of a finite volume scheme (FVS) that allows us to investigate the behavior of the solutions to \eqref{1.1}. To ensure energy conservation, particularly for  kernels with relatively small degrees of homogeneity $\theta, \gamma,\delta$, we propose an alternative weighted finite volume scheme (FVS) that incorporates specific weight functions. In addition to formulating these numerical schemes, this article provides a comprehensive analysis of the convergence of the proposed scheme. We further demonstrate that the scheme is first-order consistent, which is verified through the experimental order of convergence calculations across various test cases. {In our numerical tests, we expect energy to cascade when the degrees of homogeneity of the kinetic kernels, $\theta$, $\gamma$, and $\delta$, are large, specifically, when $\theta, \gamma, \delta \ge \frac{1}{2}$, and all coefficients $K_1$, $K_2$, and $K_3$ are nonzero.
}

\subsection{Other numerical works for coagulation-fragmentation equations }\label{Sub:IntroCoFrag}

Even though there have been quite a lot of numerical works for coagulation-fragmentation models, in most of those works, either the  collision operators only contain coagulation components   or the collision operators only contain both coagulation and fragmentation components but the fragmentation parts are linear (see, for instance \cite{degond2016numerical,filbet2004numerical}). 
Recently, a few numerical works have also been done for nonlinear   coagulation-fragmentation models (see \cite{das2024approximate,das2024improved, das2023development}). { However, the mechanisms of the models considered in those works involve collision operators that represent only a subset of our full collision operator \eqref{eq_1}. As a result, the dynamics of those models do not capture the complete scope of wave turbulence kinetics. For example, the models presented in \cite{das2024improved, das2023development} describe a pure fragmentation phenomenon, which is given by
}
\begin{align*}
	\frac{\partial f(\omega,t)}{\partial t} = \int_{0}^{\infty} \int_{\omega}^{\infty} K(\mu, \nu) B(\omega,\mu;\nu) f(\mu,t) f(\nu,t)\dd \mu \dd \nu - f(\omega,t)\int_{0}^{\infty} k(\omega,\mu) f(\mu,t) \dd \mu,
\end{align*}
for explicit kernels $B(\omega,\mu;\nu),k(\omega,\mu),k(\omega,\mu)$, which comprises of a small part of  the $9$ components displayed in \eqref{eq_1}. Those models are therefore much easier to numerically investigate. 

Therefore, the current manuscript is the {\it first work that conducts numerical studies for a fully nonlinear coagulation-fragmentation model.}

\subsection{Other numerical schemes for wave kinetic equations}\label{Sub:IntroCoFrag}	
Let us discuss  the works \cite{bell2017self,semisalov2021numerical} where numerical methods have been employed to solve various types of wave turbulence kinetic equations. In \cite{bell2017self}, the author computed the self-similar profile of the Alfven wave turbulence kinetic equation, illustrating different finite capacities. The work by \cite{semisalov2021numerical} presents a numerical technique based on Chebyshev approximation to solve the self-similar profile before the first blow-up time  for a $4-$wave kinetic equation. { The loss of energy phenomenon rigorously proved in \cite{soffer2020energy} has not been studied in these works. }


\subsection{Plan of the paper}

The remainder of this article is organized as follows: In Section \ref{Sec_2}, we formulate the finite volume scheme for the fully nonlinear coagulation-fragmentation model and introduce the newly formulated weighted FVS, proving its ability to conserve total energy. Section \ref{Sec_3} is dedicated to a detailed convergence analysis of the proposed FVS. In Section \ref{Sec_4}, we present numerical simulations to validate the theoretical results for the fully nonlinear coagulation-fragmentation model.

\section{Formulation of finite volume scheme}\label{Sec_2}

For the mathematical formulation of a numerical scheme, we need to truncate the given equation in a finite computational domain $ \mathcal{D}= \left(0,R\right]$. Let us consider the following truncated form of the  fully nonlinear coagulation-fragmentation model equation
\begin{align}\label{eq_3}
	\frac{\partial N_{\omega}}{\partial t} = S^{nc}_1\left[N_{\omega}\right] - S^{c}_2\left[N_{\omega}\right] - S^{c}_3\left[N_{\omega}\right], \qquad \omega \in \mathbb{R}_+.
\end{align} 
Here
\begin{align*}
	&S^{nc}_1\left[N_{\omega}\right] = \int_{0}^{\omega}K_1(\omega-\mu,\mu)N_{\omega-\mu}N_{\mu}\dd \mu - 2 \int_{0}^{R} K_1(\omega,\mu) N_\omega N_{\mu}\dd \mu,\\
	& S^{c}_2\left[N_{\omega}\right] = - \int_{0}^{\omega} K_2(\mu,\omega-\mu)N_\omega N_{\mu}\dd \mu + \int_{\omega}^{R} K_2(\omega,\mu-\omega)N_{\mu-\omega}N_{\mu}\dd \mu+ \int_{0}^{R-\omega} K_2(\omega,\mu) N_\omega N_{\omega+\mu}\dd \mu,\\
	& S^{c}_3\left[N_{\omega}\right] = -\int_{0}^{\omega}K_3(\mu,\omega-\mu)N_\omega N_{\omega-\mu}\dd \mu + \int_{\omega}^{R} K_3(\omega,\mu-\omega)N_{\omega}N_{\mu}\dd \mu + \int_{0}^{R-\omega}K_3(\omega,\mu) N_{\mu} N_{\mu+\omega}\dd \mu.
\end{align*}

{It is worth noting that taking the limit \( R \to \infty \) in the above truncation recovers the original equation \eqref{1.1}, provided the wave kernels are chosen appropriately. We refer the reader to the work of Stewart \cite{stewart1989global}, where such integral convergence is established in the context of a coagulation-fragmentation model.
}

{We discretize the computational domain \(\mathcal{D}\) into \(I\) cells over the interval \([0, R)\), where \(R < \infty\).
} Moreover, each of the $i-$th sub interval for $i\in \{1,2,...,I\}$, is denoted by $\Lambda_i:= \left[\omega_{i-1/2},\omega_{i+1/2}\right]$ and the cell representative of the $i-$th cell is given by
\begin{equation*}
	\omega_{1/2}=0, \quad \omega_{I^h+1/2}=R, \quad \omega_i = \frac{\omega_{i-1/2}+\omega_{i+1/2}}{2}.
\end{equation*}
Furthermore, introduce the bound $\Delta\omega$ and $\Delta\omega_{min}$ as follows  
$
\Delta\omega_{min}\le \omega_{i+1/2}-\omega_{i-1/2}:=\Delta \omega_i \le \Delta\omega.
$
The integration of the truncated  fully nonlinear coagulation-fragmentation model over each cell gives the discretized scheme in $\mathbb{R}^I$ as
\begin{equation}\label{eq_5}
	\frac{\dd \mathbf{N}}{\dd t}= \mathbf{J}\left(\mathbf{N}\right) := \sum_{k=1}^{5} \mathbf{Q}^k,\quad\text{with}\quad \mathbf{N}(0) = \mathbf{N}^{in}.
\end{equation}
Where $\ds \mathbf{N}, \mathbf{N}^{in},$ and $\mathbf{Q}^k$ for $(k=1,2...,5)$ are all in $\mathbb{R}^I$. The $i-$th component of these vectors are given by
\begin{align}
	&N_i(t) = \int_{\omega_{i-1/2}}^{\omega_{i+1/2}} N_\omega(t)\dd \omega,\quad \text{with}\quad N_i^{in}(t) =  \int_{\omega_{i-1/2}}^{\omega_{i+1/2}} N_\omega(0)\dd \omega,\label{eq_5.1}\\
	& Q^1_i(t) = \int_{\omega_{i-1/2}}^{\omega_{i+1/2}}\int_{0}^{\omega}K_1(\omega-\mu,\mu)N_{\omega-\mu}N_{\mu}\dd \mu \dd \omega,\label{eq_5.2} \\
	& Q^2_i(t) =-2 \int_{\omega_{i-1/2}}^{\omega_{i+1/2}}\int_{0}^{R} K_1(\omega,\mu) N_\omega N_{\mu}\dd \mu, \label{eq_5.3}\\
	& Q^3_i(t) =\int_{\omega_{i-1/2}}^{\omega_{i+1/2}}\left[\int_{\omega}^{R} K_2(\omega,\mu-\omega)N_{\mu-\omega}N_{\mu}\dd \mu+\int_{0}^{R-\omega}K_3(\omega,\mu) N_{\mu} N_{\mu+\omega}\dd \mu\right] \dd \omega, \label{eq_5.4}\\
	& Q^4_i(t) =-\int_{\omega_{i-1/2}}^{\omega_{i+1/2}}\left[\int_{0}^{\omega} K_2(\mu,\omega-\mu)N_\omega N_{\mu}\dd \mu+\int_{0}^{\omega}K_3(\mu,\omega-\mu)N_\omega N_{\omega-\mu}\dd \mu\right] \dd \omega,\label{eq_5.5}\\
	&Q^5_i(t) =\int_{\omega_{i-1/2}}^{\omega_{i+1/2}}\left[\int_{0}^{R-\omega} K_2(\omega,\mu) N_\omega N_{\omega+\mu}\dd \mu+\int_{\omega}^{R} K_3(\omega,\mu-\omega)N_{\omega}N_{\mu}\dd \mu\right]  \dd \omega. \label{eq_5.6}
\end{align}

Moreover, for a fully discrete formulation, the time domain need to discretized. In this regard, we split the time interval $[0,T]$ into $N$ subintervals $
\tau_n:=[t_n,t_{n+1}),$  $\text{for} \quad n \in \{0,1,...,N-1\},
$ with $t_n=n\Delta t$ and $N\Delta t=T$.
The above discretization of frequency variable $\omega$ and time variable $t$ leads to the discretize form of the collision kernels as follows; for $i,j \in \{1,...,I\}$ 
\begin{align*}
	&K_1(u,v) \approx K_1^h(u,v)= K^1_{i,j},\quad\text{when} \quad u\in \Lambda_i, ~v\in \Lambda_j,\\
	& K_2(u,v) \approx K_2^h(u,v)= K^2_{i,j},\quad\text{when} \quad u\in \Lambda_i, ~v\in \Lambda_j,\\
	& K_3(u,v) \approx K_3^h(u,v)= K^3_{i,j},\quad\text{when} \quad u\in \Lambda_i, ~v\in \Lambda_j.
\end{align*}
Where $\ds K^h_i$ be the numerical approximation of the collision kernel $\ds K_i$ $\left(\text{for } i=1,2,3\right)$. The numerical approximate values $N_i(t)$ at time $t_n$ is denoted by $N_i^n$. Therefore, the wave density function can be represented as 
\begin{align}\label{eq_3.2}
	N_\omega \approx \sum_{i=1}^{I} N_i^n\Delta \omega_i \delta\left(\omega-\omega_i\right).
\end{align}

To formulate the numerical scheme, we also need to the following sets of indices as
\begin{align*}
	&\mathcal{I}_{j,k}^i:= \{(j, k) \in \mathbb{N}\times \mathbb{N}: \omega_{i-1/2} \le \omega_j +\omega_k < \omega_{i+1/2}\},\\
	& \mathcal{J}_{j,k}^i := \{(j, k) \in \mathbb{N}\times \mathbb{N}: \omega_{i-1/2} \le \omega_j - \omega_k < \omega_{i+1/2}\},\\
	& \mathcal{K}_{j,k}^\ast:= \{(j, k) \in \mathbb{N}\times \mathbb{N}:  \omega_j +\omega_k >R\}.
\end{align*}

Using the aforementioned notation of the index sets with the approximation \eqref{eq_3.2}, the numerical \emph{Finite Volume Scheme} (FVS) can be written as follows:
\begin{align}\label{eq_6}
	N_i^{n+1}=N_i^n + \Delta t^n  &\left(\sum_{(j,k)\in \mathcal{I}_{j,k}^i} K^1_{j,k} N_j^n N_k^n \frac{\Delta \omega_j \Delta \omega_k}{\Delta \omega_i} -2 \sum_{j=1}^{I} K^1_{i,j} N_i^n N_j^n \Delta \omega_j \right.\notag\\
	&\left. + \sum_{j=i+1}^{I}  \left(K^2_{j-i,i}+K^3_{j-i,i}\right)N^n_i N^n_j \Delta \omega_j-\sum_{j=1}^{i-1} \left(K^2_{i-j,j}+K^3_{i-j,j}\right)N^n_i N^n_j \Delta \omega_j \right.\notag\\
	&\left.\quad + \sum_{(j,k)\in \mathcal{J}_{j,k}^i} \left(K^2_{j-k,k}+K^3_{j-k,k}\right) N_j^n N_k^n  \frac{\Delta \omega_j \Delta \omega_k}{\Delta \omega_i} \right) .
\end{align}

\subsection{New formulation of the energy conserving scheme}
Conserving the total energy using an FVS is one of the primary goals of this article. In this regard, in addition to verifying all the theoretical aspects of the three-wave kinetic equation \eqref{1.1}, the scheme \eqref{eq_6} is also capable of predicting the total wave action. However, our proposed scheme does not conserve the total energy of the system. This can be achieved by incorporating suitable weight functions into the formulation. Consequently, our newly formulated energy-conserving scheme takes the following form:

\begin{align}\label{eq_7}
	N_i^{n+1}=N_i^n + \Delta t^n  &\left(\sum_{(j,k)\in \mathcal{I}_{j,k}^i} K^1_{j,k} N_j^n N_k^n \frac{\Delta \omega_j \Delta \omega_k}{\Delta \omega_i} \alpha_{j,k}  -2 \sum_{j=1}^{I} K^1_{i,j} N_i^n N_j^n \Delta \omega_j \right.\notag\\
	&\left. + \sum_{j=i+1}^{I}  \left(K^2_{j-i,i}+K^3_{j-i,i}\right)N^n_i N^n_j \Delta \omega_j-\sum_{j=1}^{i-1} \left(K^2_{i-j,j}+K^3_{i-j,j}\right)N^n_i N^n_j \Delta \omega_j \beta_{i,j} \right.\notag\\
	&\left.\quad + \sum_{(j,k)\in \mathcal{J}_{j,k}^i} \left(K^2_{j-k,k}+K^3_{j-k,k}\right) N_j^n N_k^n  \frac{\Delta \omega_j \Delta \omega_k}{\Delta \omega_i} \right) .
\end{align}

Where the weights $\ds \alpha_{j,k}$ are $\beta_{i,j}$ are the weights responsible for the conservation of energy. The weights are defined as 
\begin{align*}
	\alpha_{j,k} := \left\{\begin{array}{ll}
		\frac{\omega_j + \omega_k}{\omega_i}, &\mbox{if}\quad \omega_j + \omega_k \le R, \vspace{0.2cm}\\
		0 , &\mbox{otherwise}.
	\end{array}\right.
\end{align*}
\begin{align*}
	\beta_{i,j}:= \left\{\begin{array}{ll}
		\frac{ 2\omega_j}{\omega_i}, &\mbox{if}\quad 0 < \omega_i,  \omega_j \le R, \vspace{0.2cm}\\
		0 , &\mbox{otherwise}.
	\end{array}\right. 
\end{align*}
Before proceeding to the propositions, we first impose a constraint on the kernel to ensure the mathematical validity of the numerical method. This constraint is defined as follows:
\begin{align}\label{eq_10}
	K^{i}\left(\omega,\mu\right):= \left\{\begin{array}{ll}
		K^{i}\left(\omega,\mu\right), &\mbox{if}\quad 0 < \omega+\mu\le R\text{ and } \omega,\mu>0, \vspace{0.2cm}\\
		0 , &\mbox{otherwise}.
	\end{array}\right. 
\end{align}
for $i=1,2,3$.  This constraint is applied purely for the mathematical verification of the properties of total energy conservation and the consistency of the numerical method.

\begin{Proposition}\label{Prop_1}
	The proposed finite volume scheme \eqref{eq_7} is energy conserving.
\end{Proposition} 

\begin{proof}
	The above result can be proved as follows. Multiply the discrete formulation \eqref{eq_6} by $\ds \omega_i\Delta \omega_i$ and taking sum over all i on both side, we can obtain
	\begin{align*}
		\sum_{i=1}^{I} \omega_i N_i^{n+1} \Delta \omega_i = \sum_{i=1}^{I} \omega_i N_i^{n} \Delta \omega_i + \Delta t^n T.
	\end{align*}
	Where\allowdisplaybreaks
	\begin{align*}
		T=& \sum_{i=1}^{I} \sum_{(j,k)\in \mathcal{I}_{j,k}^i} \omega_i K^1_{j,k} N_j^n N_k^n \Delta \omega_j \Delta \omega_k \alpha_{j,k} -2 \sum_{i=1}^{I} \sum_{j=1}^{I}\omega_i K^1_{i,j} N_i^n N_j^n {\Delta\omega_i}\Delta \omega_j  \\
		&\quad+\sum_{i=1}^{I} \sum_{(j,k)\in \mathcal{J}_{j,k}^i} \omega_i \left(K^2_{j-k,k}+K^3_{j-k,k}\right) N_j^n N_k^n  \Delta \omega_j \Delta \omega_k \\
		&\quad- \sum_{i=2}^{I} \sum_{j=1}^{i-1} \omega_i \left(K^2_{i-j,j}+K^3_{i-j,j}\right)N^n_i N^n_j \Delta \omega_i {\Delta\omega_j} \beta_{i,j}\\
		&\quad +  \sum_{i=1}^{I}\sum_{j=i+1}^{I} \omega_i  \left(K^2_{j-i,i}+K^3_{j-i,i}\right)N^n_i N^n_j \Delta \omega_j \Delta \omega_i.
	\end{align*}
	
	In order to prove the energy conservation we need to show that $T=0$. First we simplify the terms in right hand side of $T$.
	
	From the definition of the weight $\alpha$ we have
	\begin{align*}
		\sum_{(j,k)\in \mathcal{K}_{j,k}^\ast} K^1_{j,k} N_j^n N_k^n \Delta \omega_j \Delta \omega_k \alpha_{j,k} = 0.
	\end{align*}
	Further, $\ds \sum_{i=1}^{I} \sum_{j=1}^{I} $ represents all possible combinations of the cells $i$ and $j$. Here two cases will arise: either the sum of cell representative of the cells $i$ and $j$, i.e., $\omega_i + \omega_j$ will fall either inside or outside of the computational domain. So, we can write $\ds \sum_{j=1}^{I} = \sum_{(j,k)\in \mathcal{I}_{j,k}^i} + \sum_{(j,k)\in \mathcal{K}_{j,k}^\ast}$. Therefore, using this relation, the first term of $T$ can be written as
	\begin{align}\label{eq_8}
		\sum_{i=1}^{I} \sum_{(j,k)\in \mathcal{I}_{j,k}^i} {\omega_i} K^1_{j,k} N_j^n N_k^n \Delta \omega_j \Delta \omega_k \alpha_{j,k} = \sum_{i=1}^{I} \sum_{j=1}^{I} \left(\omega_i+\omega_j\right) K^1_{i,j} N_i^n N^n_j \Delta \omega_i \Delta \omega_j.
	\end{align}
	The third term also can be simplified in the similar manner
	\begin{align}\label{eq_9}
		\sum_{i=1}^{I} \sum_{(j,k)\in \mathcal{J}_{j,k}^i} \omega_i \left(K^2_{j-k,k}+K^3_{j-k,k}\right) N_j^n N_k^n  \Delta \omega_j \Delta \omega_k =\sum_{i=1}^{I}  \sum_{j=i+1}^{I} \omega_i \left(K^2_{j-i,i}+K^3_{j-i,i}\right) N_j^n N_i^n  \Delta \omega_j \Delta \omega_i.
	\end{align}
	Finally, using the simplified values \eqref{eq_8}, \eqref{eq_9} and the weight $\beta$ in $T$.\allowdisplaybreaks
	\begin{align*}
		T&= \sum_{i=1}^{I} \sum_{j=1}^{I} \left(\omega_i+\omega_j\right) K^1_{i,j} N_i^n N^n_j \Delta \omega_i \Delta \omega_j - 2 \sum_{i=1}^{I} \sum_{j=1}^{I}\omega_i K^1_{i,j} N_i^n N_j^n {\Delta\omega_i} \Delta \omega_j \\
		&\quad+\sum_{i=1}^{I}  \sum_{j=i+1}^{I} \omega_i \left(K^2_{j-i,i}+K^3_{j-i,i}\right) N_j^n N_i^n  \Delta \omega_j \Delta \omega_i\\
		&\quad- 2\sum_{i=2}^{I} \sum_{j=1}^{i-1} \omega_j \left(K^2_{i-j,j}+K^3_{i-j,j}\right)N^n_i N^n_j \Delta \omega_i {\Delta\omega_j} \\
		& \quad +   \sum_{i=1}^{I}\sum_{j=i+1}^{I} \omega_i \left(K^2_{j-i,i}+K^3_{j-i,i}\right)N^n_i N^n_j \Delta \omega_j \Delta \omega_i\\
		&= \sum_{i=1}^{I} \sum_{j=1}^{I} \left(\omega_i+\omega_j\right) K^1_{i,j} N_i^n N^n_j \Delta \omega_i \Delta \omega_j - 2 \sum_{i=1}^{I} \sum_{j=1}^{I}\omega_i K^1_{i,j} N_i^n N_j^n {\Delta\omega_i} \Delta \omega_j \\
		&\quad+\sum_{i=1}^{I}  \sum_{j=i+1}^{I} \omega_i \left(K^2_{j-i,i}+K^3_{j-i,i}\right) N_j^n N_i^n  \Delta \omega_j \Delta \omega_i\\
		&\quad- 2\sum_{j=1}^{I} \sum_{i=j+1}^{I} \omega_j \left(K^2_{i-j,j}+K^3_{i-j,j}\right)N^n_i N^n_j \Delta \omega_i {\Delta\omega_j} \\
		& \quad +   \sum_{i=1}^{I}\sum_{j=i}^{I} \omega_i \left(K^2_{j-i,i}+K^3_{j-i,i}\right)N^n_i N^n_j \Delta \omega_j \Delta \omega_i=0.
	\end{align*}
	Hence, the total energy of the system is conserved by the proposed scheme.
\end{proof}

\section{Convergence analysis of the FVS}\label{Sec_3}
Let us express the numerical approximate solution in the vector form $\hat{\mathbf{F}}:= \{\hat{F_1},\hat{F_2},...,\hat{F_I}\}$ where $\hat{F_i}= \hat{N_i}\Delta \omega_i$ and $\hat{N_i}$ be the approximate solution satisfying the following relation,
\begin{align}\label{con_0}
	N_\omega = \sum_{i=1}^{I} F_i  \delta\left(\omega-\omega_i\right) + \mathcal{O}\left(\Delta \omega^3\right).
\end{align}
We also denote the projected exact solution by $\ds \mathbf{F}:= \{F_1,F_2,...,F_I\}$. Then FVS \eqref{eq_6} can be rewritten in the following vector form
\begin{align}\label{con_1}
	\frac{\dd \hat{\mathbf{F}}}{\dd t}=\hat{\mathbf{J}}\left(\hat{\mathbf{F}}\right).
\end{align}
Here $\hat{\mathbf{J}}:= \{\hat{J}_1,\hat{J}_2,...,\hat{J}_I\}$ is a vector in $\mathbb{R}^I$, whose $i-$th component is given by
\begin{align}\label{con_2}
	\hat{J}_i\left(\hat{\mathbf{F}}\right):= \sum_{k=1}^{5} \hat{Q}_i^k \left(\hat{\mathbf{F}}\right).
\end{align}
Where
\begin{align}\label{con_3}
	\hat{Q}_i^1 \left(\hat{\mathbf{F}}\right):=\sum_{(j,k)\in \mathcal{I}_{j,k}^i} K^1_{j,k} \hat{F}_j(t) \hat{F}_k(t),
\end{align}
\begin{align}\label{con_4}
	\hat{Q}_i^2 \left(\hat{\mathbf{F}}\right):=-2 \sum_{j=1}^{I} K^1_{i,j} \hat{F}_i(t) \hat{F}_j(t),
\end{align}
\begin{align}\label{con_5}
	\hat{Q}_i^3 \left(\hat{\mathbf{F}}\right):=\sum_{(j,k)\in \mathcal{J}_{j,k}^i} \left(K^2_{j-k,k}+K^3_{j-k,k}\right) \hat{F}_j(t) \hat{F}_k(t),
\end{align}
\begin{align}\label{con_6}
	\hat{Q}_i^4 \left(\hat{\mathbf{F}}\right):=-\sum_{j=1}^{i-1} \left(K^2_{i-j,j}+K^3_{i-j,j}\right)\hat{F}_i(t) \hat{F}_j(t),
\end{align}
\begin{align}\label{con_7}
	\hat{Q}_i^5 \left(\hat{\mathbf{F}}\right):=\sum_{j=i+1}^{I}  \left(K^2_{j-i,i}+K^3_{j-i,i}\right)\hat{F}_i(t) \hat{F}_j(t).
\end{align}
The subsequent part of this section we consider the discrete $L^1$ norm
\begin{align*}
	\|\mathbf{F}(t)\| := \sum_{i=1}^{I} \left|F_i(t)\right|.
\end{align*}
Moreover, we consider the space $\ds \mathcal{C}^2\left(\left(0,R\right]\right)$, the space of twice continuous differentiable on $\left(0,R\right]$. For the sake of simplicity of our convergence analysis, we assume that the wave kernels are lies in the space $\ds \mathcal{C}^2\left(\left(0,R\right]\times \left(0,R\right]\right)$, that is
\begin{align}\label{con_8}
	K_i (\omega,\mu)\in \mathcal{C}^2\left(\left(0,R\right]\times \left(0,R\right]\right),\quad\text{for}\quad i=1,2,3.
\end{align}

In order to prove the convergence analysis of the above defined numerical scheme, we need to define some useful definition and results from \cite{hundsdorfer2013numerical} and \cite{linz1975convergence}, which we will use in the rest of our study.

\begin{defn}\label{def_1}
	The spatial discretization error is defined by the residual from
	substituting the exact solution $\ds \mathbf{F}:= \{F_1,F_2,...,F_I\}$ into the discrete system as
	\begin{align*}
		\sigma(t) = \frac{\dd \mathbf{F}}{\dd t} - \hat{\mathbf{J}}\left(\mathbf{F}\right).
	\end{align*}
	We will say the numerical scheme \eqref{con_1} is $p-$th order consistent if, $\Delta \omega \longrightarrow 0$
	\begin{align}\label{con_9}
		\|\sigma(t)\| = \mathcal{O}\left(\Delta\omega^p\right) \quad\text{uniformly for all}\quad 0\le t\le T.
	\end{align}
\end{defn}

Our first goal is to prove that the mapping $\ds \hat{\mathbf{J}}$ satisfies the Lipschitz condition with a Lipschitz constant $\gamma$ that is independent of the mesh. Before doing so, we will establish the uniform boundedness of the numerical solution in $L^1\left(0,T\right)$ through the following proposition.
\begin{Proposition}\label{Prop_2}
	Let the wave $K_1$, $K_2$ and $K_3$ {satisfy} the assumption \eqref{con_8}. Then there 
	{exists} a positive constant $\mathcal{A}(T)$ independent of mesh satisfying the condition
	\begin{align*}
		\|\mathbf{F}(t)\| \le \mathcal{A}(T)\quad\text{for all}\quad t \in [0,T].
	\end{align*}
\end{Proposition}
\begin{proof}
	Taking sum on the both side of equation \eqref{con_1} over $i$ from $1$ to $I$, we get
	\begin{align*}
		\frac{\dd }{\dd t}\|\mathbf{F}(t)\|=& \sum_{i=1}^{I} \sum_{(j,k)\in \mathcal{I}_{j,k}^i} { K^1_{j,k} F_j(t) F_k(t)} -2 \sum_{i=1}^{I} \sum_{j=1}^{I} K^1_{i,j} F_i(t) F_j(t)\\
		&+\sum_{i=1}^{I} \sum_{(j,k)\in \mathcal{J}_{j,k}^i} \left(K^2_{j-k,k}+K^3_{j-k,k}\right) F_j(t) F_k(t) - \sum_{i=2}^{I} \sum_{j=1}^{i-1} \left(K^2_{i-j,j}+K^3_{i-j,j}\right)F_i(t) F_j (t)\\
		& +  \sum_{i=1}^{I}\sum_{j=i+1}^{I} \left(K^2_{j-i,i}+K^3_{j-i,i}\right)F_i(t) F_j (t).
	\end{align*}
	Now proceed similar as Proposition \ref{Prop_1} to simplify the first and third term in the right hand side of the above equation, yields
	\begin{align*}
		\frac{\dd }{\dd t}\|\mathbf{F}(t)\| =& -\sum_{i=1}^{I} \sum_{j=1}^{I} K^1_{i,j} F_i(t) F_j(t) +  \sum_{i=1}^{I}  \sum_{j=i+1}^{I} \left(K^2_{j-i,i}+K^3_{j-i,i}\right) F_j(t) F_i(t)\\
		&- \sum_{i=2}^{I} \sum_{j=1}^{i-1} \left(K^2_{i-j,j}+K^3_{i-j,j}\right)F_i(t) F_j (t) +  \sum_{i=1}^{I}\sum_{j=i+1}^{I} \left(K^2_{j-i,i}+K^3_{j-i,i}\right)F_i(t) F_j (t).
	\end{align*}
	Since, we assume that $\ds K_i (\omega,\mu)\in \mathcal{C}^2\left(\left(0,R\right]\times \left(0,R\right]\right)$ for $i=1,2,3$, so there exist a constant $\mathcal{L}$ such that
	\begin{align}\label{con_10}
		\sup_{\left(\omega,\mu\right)\in \left(0,R\right]} K_i\left(\omega,\mu\right) \le \mathcal{L}.
	\end{align}
	This estimation reduces the above identity into the following differential inequality
	\begin{align*}
		\frac{\dd }{\dd t}\|\mathbf{F}(t)\| \le 2 \mathcal{L} \|\mathbf{F}(t)\|^2, \quad \text{for all}\quad 0\le t\le T.
	\end{align*}
	{Hence, the uniform boundedness result directly follows from this inequality}.
\end{proof}
\begin{Proposition}\label{Prop_3}
	Assume that the wave kernels $K_1$, $K_2$ and $K_3$ {lie} in the space $\mathcal{C}^2\left(\left(0,R\right]\times \left(0,R\right]\right)$. Then there {exists} a positive constant $\gamma$ independent of mesh such that
	\begin{align}\label{con_11}
		\|\hat{\mathbf{J}}\left(\mathbf{F}\right)-\hat{\mathbf{J}}\left(\hat{\mathbf{F}}\right)\| \le \gamma \|\mathbf{F}-\hat{\mathbf{F}}\|,\quad \text{for all}\quad \mathbf{F}, \hat{\mathbf{F}}\in \mathbb{R}^I.
	\end{align}
\end{Proposition}
\begin{proof}
	In order to prove the mapping $\hat{\mathbf{J}}$ defined in \eqref{con_2} satisfy the Lipscitz condition, we will show that each of the term $\hat{Q}^k$ $(k=1,...,5)$ is satisfying the Lipscitz condition. Let,  $\mathbf{F}, \hat{\mathbf{F}}\in \mathbb{R}^I$,
	\begin{align*}
		\|\hat{Q}^1\left(\mathbf{F}\right)-\hat{Q}^1\left(\hat{\mathbf{F}}\right)\| &= \sum_{i=1}^{I}\left|\hat{Q}^1_i\left(\mathbf{F}\right)-\hat{Q}_i^1\left(\hat{\mathbf{F}}\right)\right|
		= \sum_{i=1}^{I}\sum_{(j,k)\in \mathcal{I}_{j,k}^i} \left|K^1_{j,k}\right|\left|F_j F_k-\hat{F}_j\hat{F}_k\right|.
	\end{align*}
	Now one can observe the following fact 
	\begin{align}\label{con_12}
		\left|F_j F_k-\hat{F}_j\hat{F}_k\right|= \frac{1}{2}\left|\left(F_j+\hat{F}_j\right)\left(F_k-\hat{F}_k\right) + \left(F_j-\hat{F}_j\right)\left(F_k+\hat{F}_k\right)\right|.
	\end{align}
	Using the Proposition \ref{Prop_2} together with condition \eqref{con_10} and identity \eqref{con_12} on the above norm reduce to
	\begin{align}
		\|\hat{Q}^1\left(\mathbf{F}\right)-\hat{Q}^1\left(\hat{\mathbf{F}}\right)\| &\le\frac{\mathcal{L}}{2}\sum_{j=1}^{I}\sum_{k=1}^{I}\left|\left(F_j+\hat{F}_j\right)\left(F_k-\hat{F}_k\right) + \left(F_j-\hat{F}_j\right)\left(F_k+\hat{F}_k\right)\right| \notag\\
		&\le 2\mathcal{L} \mathcal{A}(T)\|\mathbf{F}-\hat{\mathbf{F}}\|. \notag
	\end{align}
	Similar process can be follow to show the remaining terms that is $\hat{Q}^k$ $(k=2,...,5)$ are also satisfies the Lipscitz condition. And substitute all these estimations in the following identity
	\begin{align*}
		\|\hat{\mathbf{J}}\left(\mathbf{F}\right)-\hat{\mathbf{J}}\left(\hat{\mathbf{F}}\right)\|  = \sum_{k=1}^{5}\|\hat{Q}^k\left(\mathbf{F}\right)-\hat{Q}^k\left(\hat{\mathbf{F}}\right)\|.
	\end{align*}
	we can get a constant $\gamma$ which is depending only on $T$ and $\mathcal{L}$ not on the choice of mesh such that the inequality \eqref{con_11} holds.
\end{proof}
Now we are in the stage where we can start proving our main convergence theorem.
\begin{Theorem}\label{thm_3}
	{Let the wave interaction kernels $K_i$ $(i=1,2,3)$ are belong to the space} $\mathcal{C}^2\left(\left(0,R\right]\times \left(0,R\right]\right)$. Then the proposed finite volume scheme \eqref{con_1} is nonnegative and consistent with first order consistency. Consequently, the numerical scheme \eqref{con_1} is convergent and the order of convergence is the same of order of consistency.
\end{Theorem}

\begin{proof}	
	To prove the convergence of the proposed finite volume scheme \eqref{con_1}, our next goal is to demonstrate that the numerical scheme is both nonnegative and consistent. To begin, we show that the scheme is nonnegative.
	
	{\it Nonnegativity:}
	Consider any nonnegative approximate solution $\mathbf{\hat{F}}=\{\hat{F_1},\hat{F_2},...,\hat{F_I}\}$ for which the $i-$th component $\hat{F}_i=0$. Then from the definition of the numerical fluxes \eqref{con_3}-\eqref{con_7}, we can get
	$
	\hat{Q}_i^1 \left(\hat{\mathbf{F}}\right) \ge 0,~ \hat{Q}_i^2 \left(\hat{\mathbf{F}}\right) =0,~ \hat{Q}_i^3 \left(\hat{\mathbf{F}}\right)\ge 0,~ \hat{Q}_i^4 \left(\hat{\mathbf{F}}\right)=0\quad\text{and}\quad\hat{Q}_i^5 \left(\hat{\mathbf{F}}\right)=0.
	$
	Therefore, using these relations in \eqref{con_2}, we get $\hat{J}_i\left(\hat{\mathbf{F}}\right)\ge 0$ for any $i \in\{1,2,...,I\}$, moreover, $\mathbf{\hat{J}}$ is Lipscitz from Proposition \ref{Prop_3}. Then thanks to the Theorem $4.1$ of \cite{hundsdorfer2013numerical}, we can say that our proposed FVS is nonnegative.
	
	{\it Consistency:}
	To prove the consistency of the proposed scheme \eqref{con_1}, we need to calculate the discretization error for each of the terms \eqref{con_3}-\eqref{con_7}. Let us start from the first term and then change the order of the integration\allowdisplaybreaks
	\begin{align*}
		Q^1_i(t) &= \int_{\omega_{i-1/2}}^{\omega_{i+1/2}}\int_{0}^{\omega}K_1(\omega-\mu,\mu)N_{\omega-\mu}(t)N_{\mu}(t)\dd \mu \dd \omega\notag\\
		&= \left[\int_{\omega_{i-1/2}}^{\omega_{i+1/2}}\int_{0}^{\omega_{i-1/2}}+\int_{\omega_{i-1/2}}^{\omega_{i+1/2}}\int_{\omega_{i-1/2}}^\omega\right]K_1(\omega-\mu,\mu)N_{\omega-\mu}(t)N_{\mu}(t)\dd \mu \dd \omega \notag \\
		&= \left[\int_{\omega_{i-1/2}}^{\omega_{i+1/2}}\sum_{j=1}^{i-1}\int_{\omega_{j-1/2}}^{\omega_{j+1/2}}+\int_{\omega_{i-1/2}}^{\omega_{i+1/2}}\int_{\omega_{i-1/2}}^\omega\right]K_1(\omega-\mu,\mu)N_{\omega-\mu}(t)N_{\mu}(t)\dd \mu \dd \omega.
	\end{align*}
	On applying the relation \eqref{con_0} for the midpoint quadrature rule on the above equation and we get $Q^1_i(t)$ in the following form,
	\begin{align*}
		\mathcal{Q}^1_i 
		& = \sum_{j=1}^{i-1}\int_{\omega_{i-1/2}}^{\omega_{i+1/2}}\int_{\omega_{j-1/2}}^{\omega_{j+1/2}} K_1(\omega-\mu,\mu) \sum_{k=1}^{I} F_k(t) \delta\left(\omega-\mu-\omega_k\right)\sum_{l=1}^{I} F_l(t) \delta\left(\mu-\omega_l\right) \dd \mu \dd \omega \\
		& \quad + \int_{\omega_{i-1/2}}^{\omega_{i+1/2}}\int_{\mu}^{\omega_{i+1/2}} K_1(\omega-\mu,\mu) \sum_{k=1}^{I} F_k(t) \delta\left(\omega-\mu-\omega_k\right)\sum_{l=1}^{I} F_l(t) \delta\left(\mu-\omega_l\right) \dd \omega \dd \mu + \mathcal{O}\left(\Delta \omega^3\right) \\
		& = \sum_{j=1}^{i-1} F_j (t) \int_{\omega_{i-1/2}}^{\omega_{i+1/2}} K_1(\omega-\omega_j,\omega_j)  \sum_{k=1}^{I} F_k(t) \delta\left(\omega-\omega_j-\omega_k\right) {\dd \omega} \\
		& \qquad +F_i(t)  \int_{\omega_i}^{\omega_{i+1/2}} K_1(\omega-\omega_i,\omega_i)  \sum_{k=1}^{I} F_k(t) \delta\left(\omega-\omega_i-\omega_k\right)  {\dd \omega} + \mathcal{O}\left(\Delta \omega^3\right) \\
		& = \sum_{j=1}^{i-1} \sum_{\omega_{i-1/2}\le \omega_j+\omega_k <\omega_{i+1/2}} K^1_{j,k} F_j(t) F_k(t)  + F_i(t) \Delta \omega_i \sum_{\omega_i+\omega_k<\omega_{i+1/2}} K^1_{i,k} F_k(t)+ \mathcal{O}\left(\Delta \omega^3\right) \\
		& = \sum_{(j,k)\in \mathcal{I}_{j,k}^i} K^1_{j,k} F_j(t) F_k(t)  + \mathcal{O}\left(\Delta \omega^3\right)\\
		&= \hat{Q}_i^1 \left(\mathbf{F}\right) + \mathcal{O}\left(\Delta \omega^3\right).
	\end{align*}
	
	Using the simple application of midpoint quadrature rule, we can discretize the second term as follows\allowdisplaybreaks
	\begin{align*}
		\mathcal{Q}^2_i&= -2 \int_{\omega_{i-1/2}}^{\omega_{i+1/2}} \sum_{j=1}^{I} \int_{\omega_{j-1/2}}^{\omega_{j+1/2}}K_1(\omega,\mu)N_\omega(t) N_\mu(t) \dd \mu \dd \omega\\
		& = -2 \int_{\omega_{i-1/2}}^{\omega_{i+1/2}}\sum_{k=1}^{I} F_k(t) \delta\left(\omega-\omega_k\right) \sum_{j=1}^{I} \int_{\omega_{j-1/2}}^{\omega_{j+1/2}}K_1(\omega,\mu){\sum_{l=1}^{I} F_l(t) \delta\left(\mu-\omega_l\right)} \dd \mu \dd \omega + \mathcal{O}\left(\Delta \omega^3\right)\\
		&= -2 \sum_{j=1}^{I} K^1_{i,j} F_i(t) F_j(t)  + \mathcal{O}\left(\Delta \omega^3\right)\\
		&= \hat{Q}_i^2 \left(\mathbf{F}\right) + \mathcal{O}\left(\Delta \omega^3\right).
	\end{align*}
	
	To find the discretization error in the term $\mathcal{Q}^3_i$, we first substitute $\ds \mu = \mu'-\omega $ in the second integral of \eqref{eq_5.4} then replace $\mu'$ by $\mu$.
	\begin{align*}
		\mathcal{Q}^3_i(t)&=\int_{\omega_{i-1/2}}^{\omega_{i+1/2}}\int_{\omega}^{R}\left[ K_2(\omega,\mu-\omega)+K_3(\omega,\mu-\omega) \right]N_{\mu-\omega}N_{\mu}{\dd \mu}\dd \omega.
	\end{align*}
	Then proceed similar as $\mathcal{Q}^1_i(t)$ above, the reduced form of $\mathcal{Q}^3_i(t)$ can be written as follows
	\begin{align*}
		\mathcal{Q}^3_i
		& = \int_{\omega_{i-1/2}}^{\omega_{i+1/2}} \sum_{j=i+1}^{I} \int_{\omega_{j-1/2}}^{\omega_{j+1/2}}\left[K_2(\omega,\mu-\omega)+K_3(\omega,\mu-\omega) \right] N_{\mu-\omega}N_{\mu}\dd \mu \dd \omega\\
		&+ \int_{\omega_{i-1/2}}^{\omega_{i+1/2}} \int_{\omega}^{\omega_{i+1/2}}\left[K_2(\omega,\mu-\omega)+K_3(\omega,\mu-\omega) \right] N_{\mu-\omega}N_{\mu}\dd \mu \dd \omega\\
		& = \sum_{j=i+1}^{I} F_j(t) \int_{\omega_{i-1/2}}^{\omega_{i+1/2}} \left[K_2(\omega,\omega_j-\omega)+K_3(\omega,\omega_j-\omega)\right] \sum_{k=1}^{I} F_k(t) \delta\left(\omega_j-\omega-\omega_k\right)\dd \omega \\
		&\qquad + F_i(t)  \int_{\omega_{i-1/2}}^{\omega_i} K_2(\omega,\omega_i-\omega)\sum_{k=1}^{I}F_k(t)\delta(\omega_i-\omega-\omega_k)\dd \omega+ \mathcal{O}\left(\Delta \omega^3\right)   \\
		& = \sum_{j=i+1}^{I} \sum_{\omega_{i-1/2} \le \omega_j -\omega_k < \omega_{i+1/2}} \left[K^2_{j-k,k}+K^3_{j-k,k}\right] F_j(t) F_k(t)\\
		&\qquad+ \sum_{\omega_{i-1/2}\le \omega_i - \omega_k}  \left[K^2_{i-k,k}+K^3_{i-k,k}\right] F_i(t) F_k(t)  + \mathcal{O}\left(\Delta \omega^3\right)\\
		& ={ \sum_{(j,k)\in \mathcal{J}_{j,k}^i} K^2_{j-k,k} F_j(t) F_k(t)} + \mathcal{O}\left(\Delta \omega^3\right)\\
		&=  \hat{Q}_i^3 \left(\mathbf{F}\right) + \mathcal{O}\left(\Delta \omega^3\right).
	\end{align*}
	
	To obtain the discretization error of the fourth term we substitute $\ds \mu = \omega-\mu' $ in the second integral of \eqref{eq_5.5} then replace $\mu'$ by $\mu$.
	\begin{align*}
		\mathcal{Q}^4_i =& -\int_{\omega_{i-1/2}}^{\omega_{i+1/2}} \int_{0}^{\omega} \left[K_2(\mu,\omega-\mu)+K_3(\mu,\omega-\mu)\right]N_\omega(t) N_{\mu}(t)\dd \mu \dd \omega \\
		& ={ F_i(t) \int_{0}^{\omega_i}\left[K_2(\mu,\omega_i-\mu)+K_3(\mu,\omega_i-\mu)\right]\sum_{k=1}^{I} F_k(t) \delta\left(\mu-\omega_k\right)\dd \mu } + \mathcal{O}\left(\Delta \omega^3\right)\\
		& = - F_i(t)\sum_{j=1}^{i-1}\int_{\omega_{j-1/2}}^{\omega_{j+1/2}} \left[K_2(\mu,\omega_i-\mu)+K_3(\mu,\omega_i-\mu)\right]\sum_{k=1}^{I} F_k(t) \delta\left(\mu-\omega_k\right)\dd \mu  \\
		&\quad- F_i(t) \int_{\omega_{i-1/2}}^{\omega_i} \left[K_2(\mu,\omega_i-\mu)+K_2(\mu,\omega_i-\mu)\right] \sum_{k=1}^{I} F_k(t) \delta\left(\mu-\omega_k\right)\dd \mu + \mathcal{O}\left(\Delta \omega^3\right) \\
		&= -\sum_{j=1}^{i-1} \left[K^2_{i-j,j}+K^3_{i-j,j}\right]F_i F_j\\
		&\quad - F_i(t) \int_{\omega_{i-1/2}}^{\omega_i} \left[K_2(\mu,\omega_i-\mu)+K_3(\mu,\omega_i-\mu)\right] \sum_{k=1}^{I} F_k(t) \delta\left(\mu-\omega_k\right)\dd \mu + \mathcal{O}\left(\Delta \omega^3\right)\\
		& = \hat{Q}_i^4 \left(\mathbf{F}\right)- E_1 + \mathcal{O}\left(\Delta \omega^3\right).
	\end{align*}
	Where
	\begin{align*}
		E_1 = F_i(t) \int_{\omega_{i-1/2}}^{\omega_i} \left[K_2(\mu,\omega_i-\mu)+K_3(\mu,\omega_i-\mu)\right] \sum_{k=1}^{I} F_k(t) \delta\left(\mu-\omega_k\right)\dd \mu.
	\end{align*}
	Now apply the right end quadrature rule on the integral of $E_1$ and using the condition \eqref{eq_10} on the wave interaction kernela $K_2$ and $K_3$. Then the integral term $E_1$ reduce to $
	E_1= \mathcal{O}\left(\Delta \omega^2\right).
	$
	Therefore,
	$
	\mathcal{Q}^4_i \left(\mathbf{F}\right)= \hat{Q}_i^4 \left(\mathbf{F}\right)+ \mathcal{O}\left(\Delta \omega^2\right).
	$
	
	For the final term, we first substitute $\ds \mu'=\omega+\mu$ on the first integral of \eqref{eq_5.6} and replace $\mu'$ by $\mu$.
	\begin{align*}
		\mathcal{Q}^5_i &= \int_{\omega_{i-1/2}}^{\omega_{i+1/2}} \int_{\omega}^{R} \left[K_2(\omega,\mu-\omega)+K_3(\omega,\mu-\omega)\right]N^n_\omega N^n_\mu \dd \mu \dd \omega \\
		& = F_i(t)\sum_{j=i+1}^{I} \int_{\omega_{j-1/2}}^{\omega_{j+1/2}}\left[K_2(\omega_i,\mu-\omega_i)+K_2(\omega_i,\mu-\omega_i)\right] \sum_{k=1}^{I} F_k(t) \delta\left(\mu-\omega_k\right)\dd \mu \\
		& \quad+ F_i(t)\int_{\omega_i}^{\omega_{i+1/2}} \left[K_2(\omega_i,\mu-\omega_i)+K_3(\omega_i,\mu-\omega_i)\right]\sum_{k=1}^{I} F_k(t) \delta\left(\mu-\omega_k\right)\dd \mu + \mathcal{O}\left(\Delta \omega^3\right) \\
		&= \sum_{j=i}^{I} \left[K^2_{j-i,i}+K^3_{j-i,i}\right]F_i(t) F_j \\
		&\quad + F_i(t)\int_{\omega_i}^{\omega_{i+1/2}} \left[K_2(\omega_i,\mu-\omega_i)+K_3(\omega_i,\mu-\omega_i)\right]\sum_{k=1}^{I} F_k(t) \delta\left(\mu-\omega_k\right)\dd \mu + \mathcal{O}\left(\Delta \omega^3\right) \\
		& = \hat{Q}_i^5 \left(\mathbf{F}\right) +E_2 + \mathcal{O}\left(\Delta \omega^3\right).
	\end{align*}
	Where
	\begin{align*}
		E_2 =  F_i(t)\int_{\omega_i}^{\omega_{i+1/2}} \left[K_2(\omega_i,\mu-\omega_i)+K_3(\omega_i,\mu-\omega_i)\right]\sum_{k=1}^{I} F_k(t) \delta\left(\mu-\omega_k\right)\dd \mu.
	\end{align*}
	Apply the right end quadrature rule on the integral of $E_1$ and using the condition \eqref{eq_10} on the wave interaction {kernels} $K_2$ and $K_3$. Then the integral term $E_1$ reduce to 
	$E_2 = \mathcal{O}\left(\Delta \omega^2\right)$.
	Finally we get the discretization error of the last term as follow
	$
	\mathcal{Q}^5_i \left(\mathbf{F}\right)= \hat{Q}_i^5 \left(\mathbf{F}\right)+ \mathcal{O}\left(\Delta \omega^2\right).
	$
	
	Now using all these obtained discretization error value in the Definition \ref{def_1}, we can obtain the spatial discretization error
	\begin{align*}
		\sigma_i = \left|J_i\left(\mathbf{F}\right)- \hat{J_i}\left({\mathbf{F}}\right)\right|= \sum_{k=1}^{5}\left|Q^k_i\left(\mathbf{F}\right)- \hat{Q}^k_i\left({\mathbf{F}}\right)\right|= \mathcal{O}\left(\Delta \omega^2\right).
	\end{align*}
	This further implies 
	$
	\|\sigma\|= \sum_{i=1}^{I}\left|\sigma_i(t)\right|=\mathcal{O}\left(\Delta \omega\right).
	$
	
	{\it Convergence:}
	To conclude the convergence proof, we combine these above results with the Proposition \ref{Prop_3}, which fulfill all the necessary conditions of the convergence analysis of \cite{linz1975convergence}. Hence, the proposed finite volume scheme \eqref{con_1} is convergent. Moreover, the scheme is first order convergent as same as the order of consistency.	
\end{proof}

\section{Numerical test}\label{Sec_4}

In this section, we consider some test problems to compare the efficiency of the newly proposed finite volume schemes \eqref{eq_6} and \eqref{eq_7} with the theoretical results of \cite{soffer2020energy}. We solve the fully nonlinear coagulation-fragmentation model using both schemes \eqref{eq_6} and \eqref{eq_7} for different initial conditions. For the test problems, we consider the collision kernels of the form $K_1(\omega,\mu)= (\omega\mu)^\theta, K_2(\omega,\mu)= (\omega\mu)^\gamma,$ and $K_3(\omega,\mu) = (\omega\mu)^\delta$, where $\theta, \gamma$ and $\delta$ are non-negative parameters referred to as the degrees of homogeneity of the collision kernels. In each test case, we either vary the degree of homogeneity of the collision kernels while holding the truncation parameter fixed at $R=100$, or we vary the truncation parameter $R$ while keeping the degree of homogeneity fixed at $\theta=\gamma=\delta=1$. Moreover, while changing the truncation parameter, we also examine the behavior of the solution under varying or fixed numbers of grid points.

In each of the test problems, the computation time domain is taken as $\ds \left[0, 1000\right]$ with the time step $\Delta t= 0.1$, for the final time $T=1000$. Moreover, all the tests were performed on a uniform grid. We observe that the choice of the step length $h$ also depends on the initial condition $N_\omega^{in}$. The first test uses a step length $h=0.5$ for the weighted scheme and $h=1$ for the finite volume scheme without any weight function. However, the non-weighted scheme also performed similarly for $h=0.5$. In the second test problem, we chose $h=0.5$ to implement our proposed FVS.

\subsection{Test case I:}\label{test_1}
In the first test problem, we consider an \emph{analytical initial condition} is given by 
\begin{align}\label{2.1}
	N_\omega^{in} = 1.25 \omega e^{-100\left(\omega-0.25\right)^2}, \qquad \omega\ge 0.
\end{align}
We begin implementing the FVS \eqref{eq_6} with the specific case where the degree of homogeneity satisfies $\theta=\gamma=\delta$. Specifically, we consider the nonlinear collision kernels given by
\begin{align}\label{2.3}
	K_1(\omega,\mu)=K_2(\omega,\mu)=K_3(\omega,\mu) = (\omega\mu)^\theta
\end{align}
and the computation domain as $ [0,100]$. With the above setup, we plot the initial state $N^{in}_\omega$ in Figure \ref{f1_1} and the final state $N_\omega(T)$ in Figure \ref{f1_2} corresponding to the collision kernels for $\theta=1$ in \eqref{2.3}. According to the theoretical results, the energy on any finite interval tends to $0$; as follows:
\begin{align}\label{2.2}
	\lim\limits_{t \to \infty} N_\omega(t) \chi_{[0,R]}(\omega) = 0.
\end{align} 
Both Figures \ref{f1_1} and \ref{f1_2} successfully verify the theoretical result \eqref{2.2} proved in \cite{soffer2020energy}.

\begin{figure}[htp]
	\begin{center}
		\begin{subfigure}{.35\textwidth}
			\centering
			\includegraphics[width=1.0\textwidth]{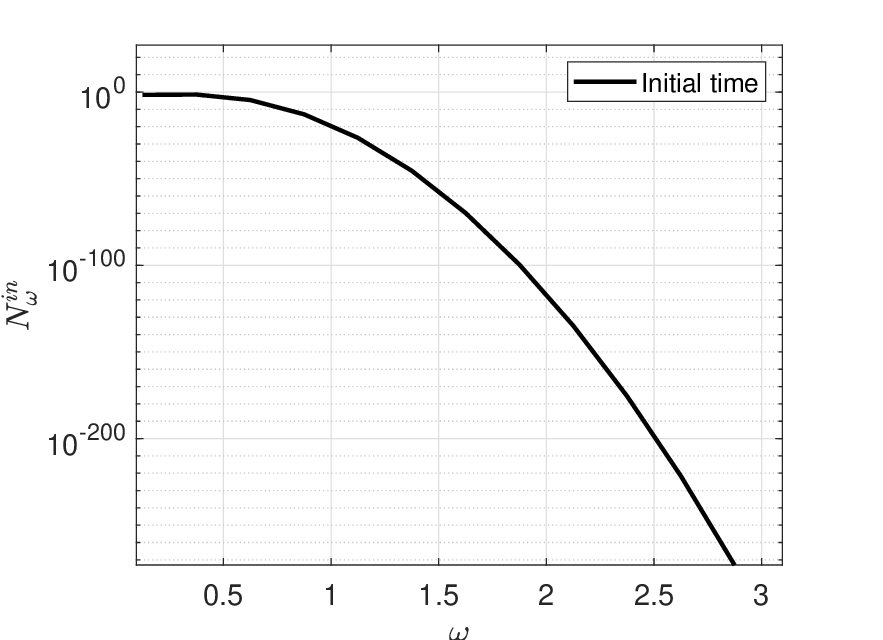}
			\caption{Initial condition}
			\label{f1_1}
		\end{subfigure}\hspace{1.2cm}
		\begin{subfigure}{.35\textwidth}
			\centering
			\includegraphics[width=1.0\textwidth]{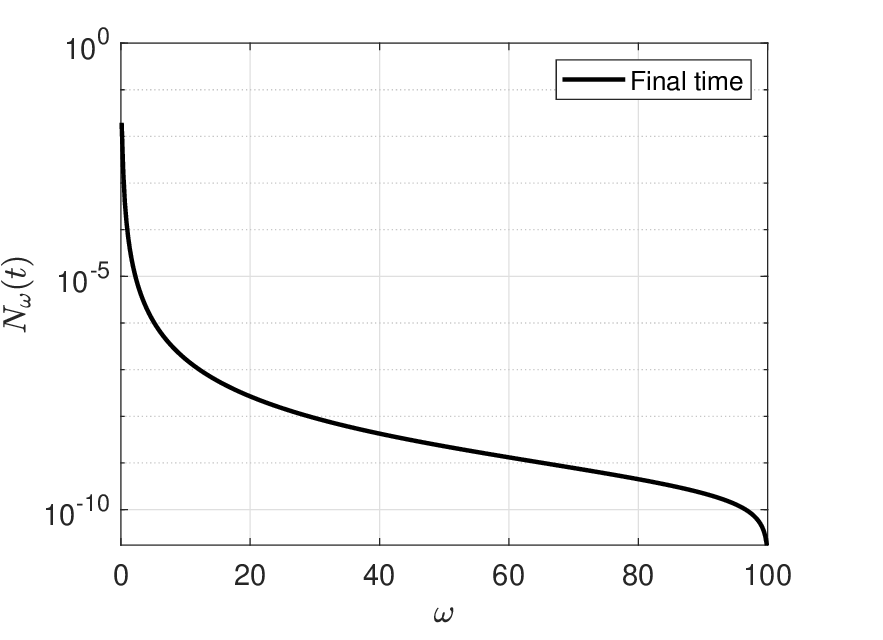}
			\caption{Final condition}
			\label{f1_2}
		\end{subfigure}
		\caption{Evolution of wave density $N_\omega$ at initial and final time.}
		\label{f1}
	\end{center}
\end{figure}

Under the consideration of the initial data \eqref{2.1} with the collision kernels \eqref{2.3}, the time evolution of the first three moments $M_1(t)$, $M_2(t)$ and $M_3(t)$ is plotted for different degrees of the collision kernel: $\theta= 1, 0.85, 0.75$. The numerical approximations in Figure \ref{f2} also support the theoretical result \eqref{2.2}, indicating that the wave density decays to $0$ with respect to the dimensionless time $t$. 
\begin{figure}[htp]
	\begin{subfigure}{.3\textwidth}
		\centering
		\includegraphics[width=1.0\textwidth]{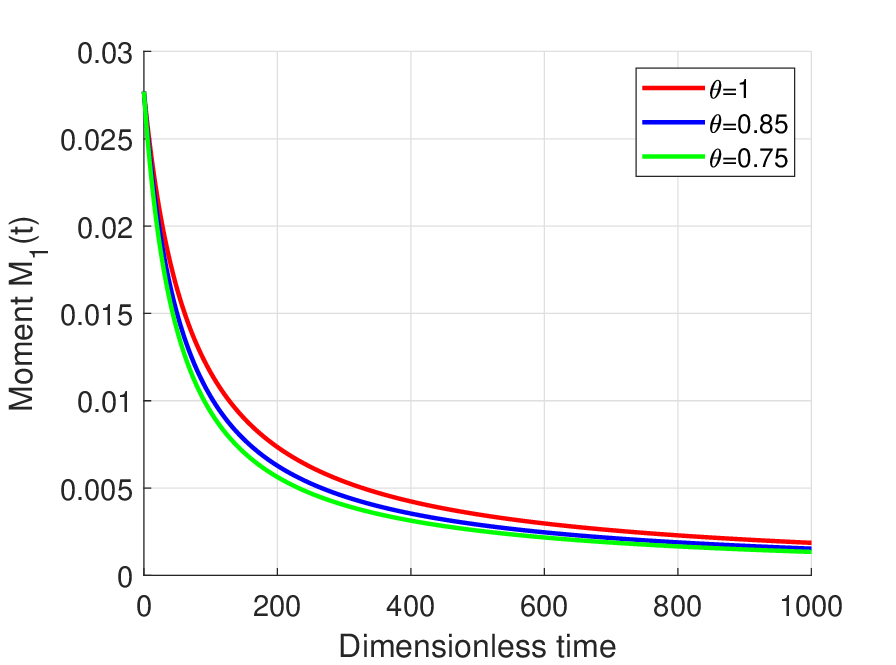}
		\caption{}
		\label{f2.1}
	\end{subfigure}
	\begin{subfigure}{.3\textwidth}
		\centering
		\includegraphics[width=1.0\textwidth]{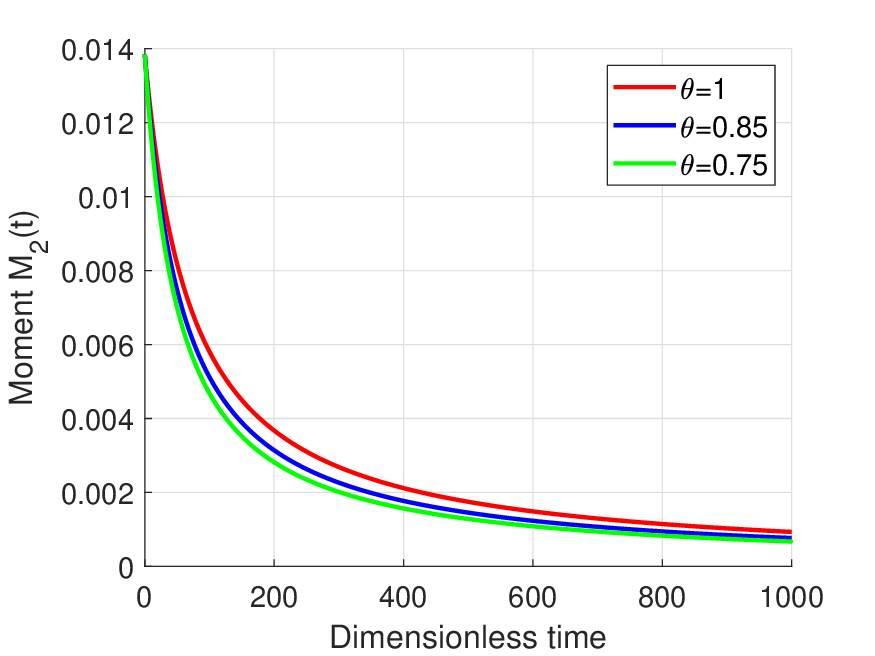}
		\caption{}
		\label{f2.2}
	\end{subfigure}
	\begin{subfigure}{.3\textwidth}
		\centering
		\includegraphics[width=1.0\textwidth]{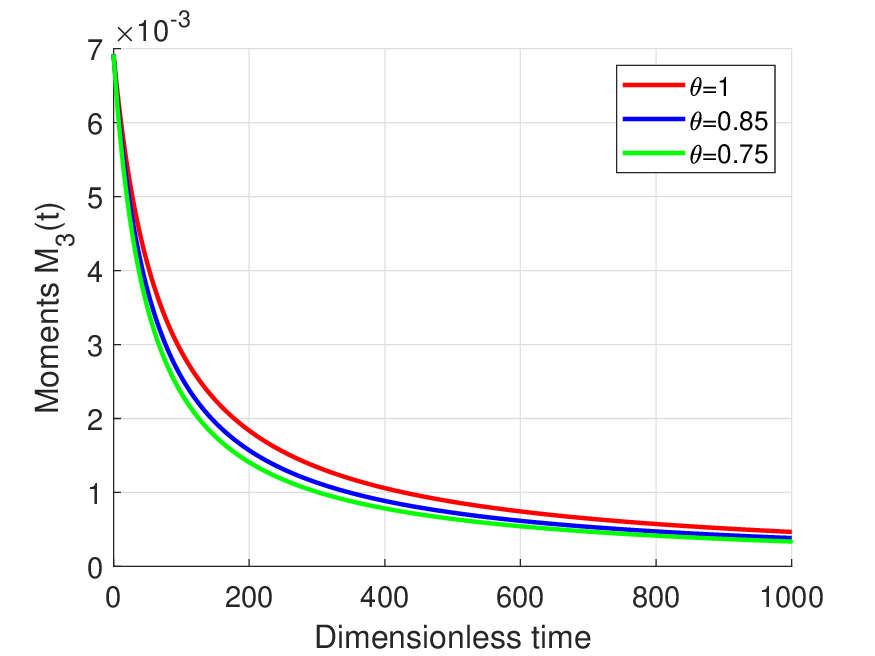}
		\caption{}
		\label{f2.3}
	\end{subfigure}
	\caption{Time evolution of the (a) first, (b) second, and (c) third moments for different degrees of homogeneity $\theta$ with the same truncation parameter $R$.}
	\label{f2}
\end{figure}

As observed from Figure \ref{f2.1}, the total energy is not conserved throughout the time domain. To investigate the theoretical decay rate of the total energy in $\left[0,\infty\right)$, Figure \ref{f3} shows the total energy for different degrees $\theta$ together with the estimation given in \eqref{1.15}.	
\begin{figure}[htp]
	\begin{center}
		\begin{subfigure}{.3\textwidth}
			\centering
			\includegraphics[width=1.0\textwidth]{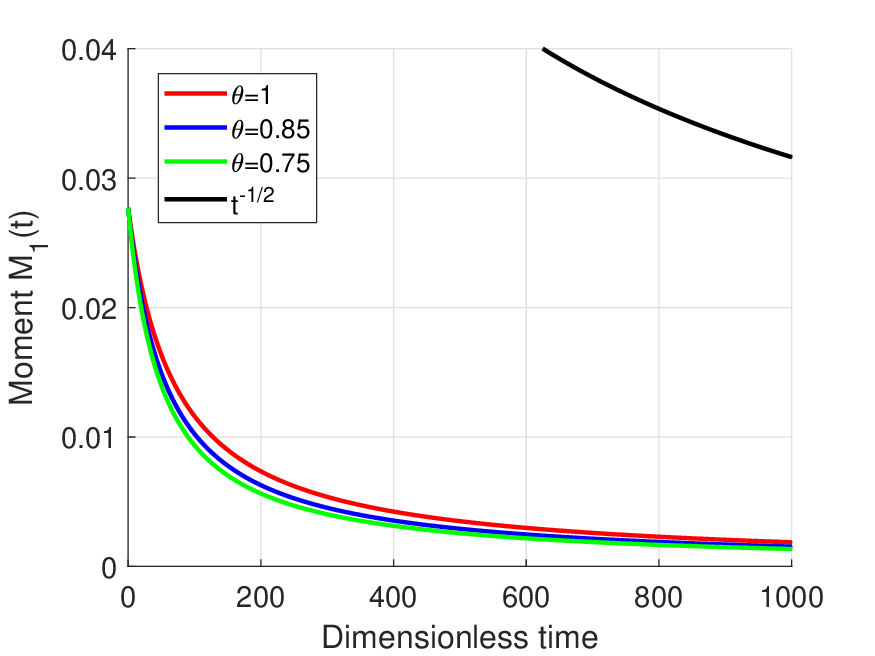}
			\caption{In linear scale}
			\label{f3_1}
		\end{subfigure}
		\begin{subfigure}{.3\textwidth}
			\centering
			\includegraphics[width=1.0\textwidth]{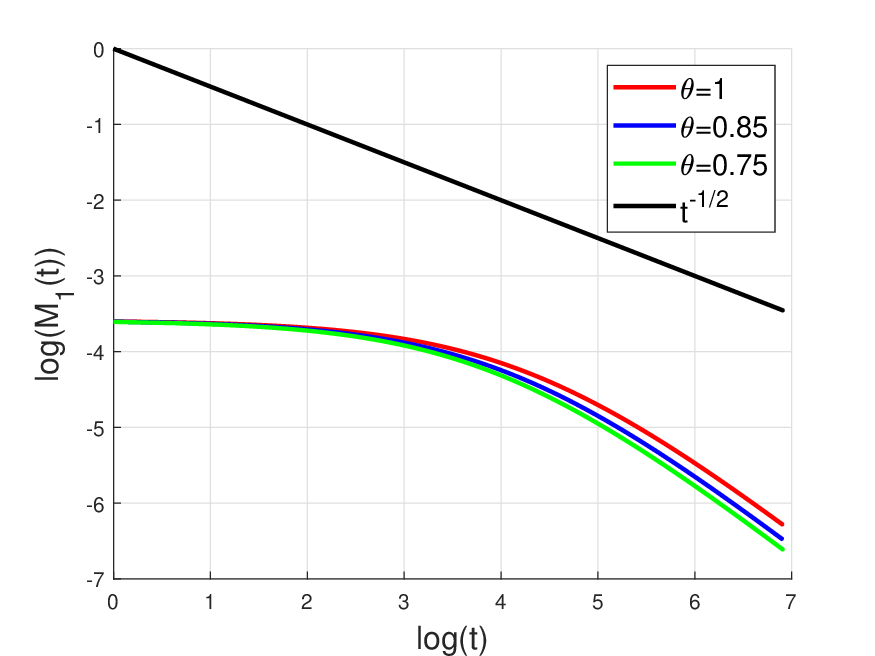}
			\caption{In logarithmic scale}
			\label{f3_2}
		\end{subfigure}
		\caption{Decay of total energy for different values of $\theta$.}
		\label{f3}
	\end{center}
\end{figure}
textcolor{blue}{Figure~\ref{f3} demonstrates that the total energy of the system cascades over time, resulting in an outflow of energy from the computational domain. Moreover, the theoretical prediction \eqref{1.15} for the decay rate of the total energy shows very good agreement with the numerical results, where the slope of the decay curve lies just below that of the reference line corresponding to \(\ds \frac{1}{\sqrt{t}}\).
}

With the same consideration of the initial data \eqref{2.1} and collision kernels \eqref{2.3}, we plot the time evolution of three different moments, $M_1(t)$, $M_2(t)$ and $M_3(t)$ for three different truncation parameters, $R=50,100$ and $120$ for a fixed degree of homogeneity $\theta=1$. However, note that in this experiment, we maintain the same mesh length by increasing the number of grid points. From Figure \ref{f4}, we can observe that the solution is not affected by the change in truncation parameter $R$ for the same value of $h$.
\begin{figure}[htp]
	\begin{subfigure}{.3\textwidth}
		\centering
		\includegraphics[width=1.0\textwidth]{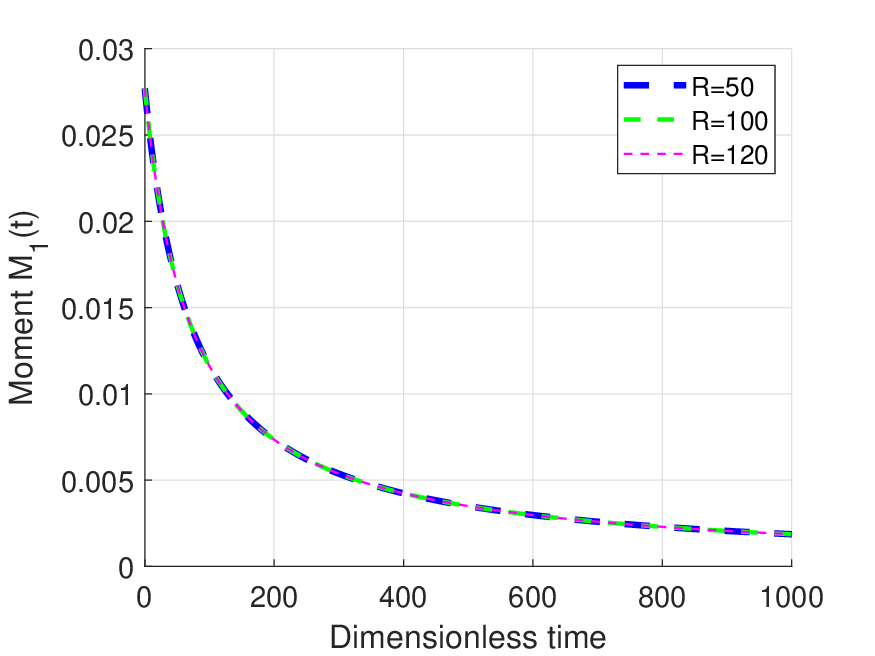}
		\caption{}
		\label{f4.1}
	\end{subfigure}
	\begin{subfigure}{.3\textwidth}
		\centering
		\includegraphics[width=1.0\textwidth]{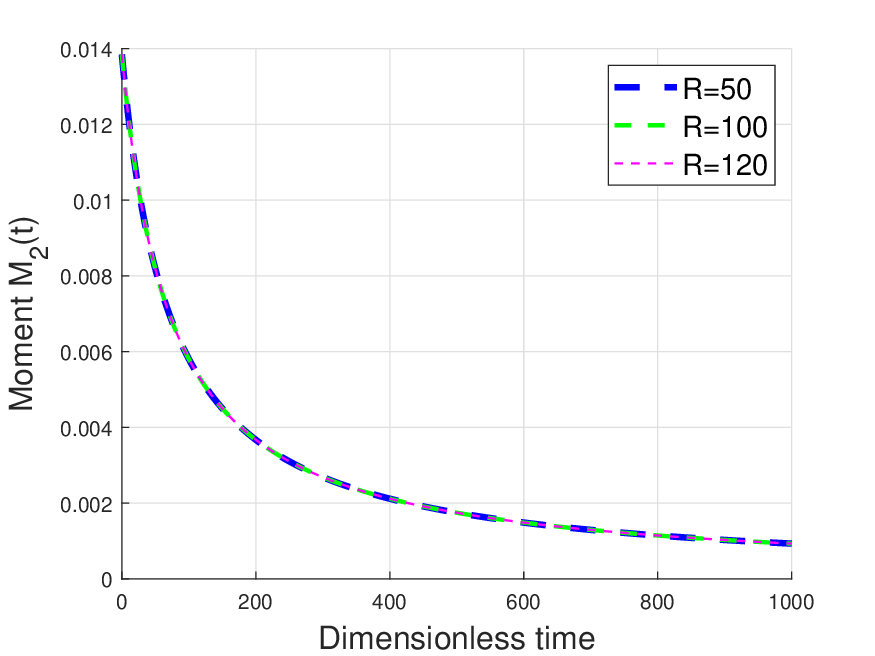}
		\caption{}
		\label{f4.2}
	\end{subfigure}
	\begin{subfigure}{.3\textwidth}
		\centering
		\includegraphics[width=1.0\textwidth]{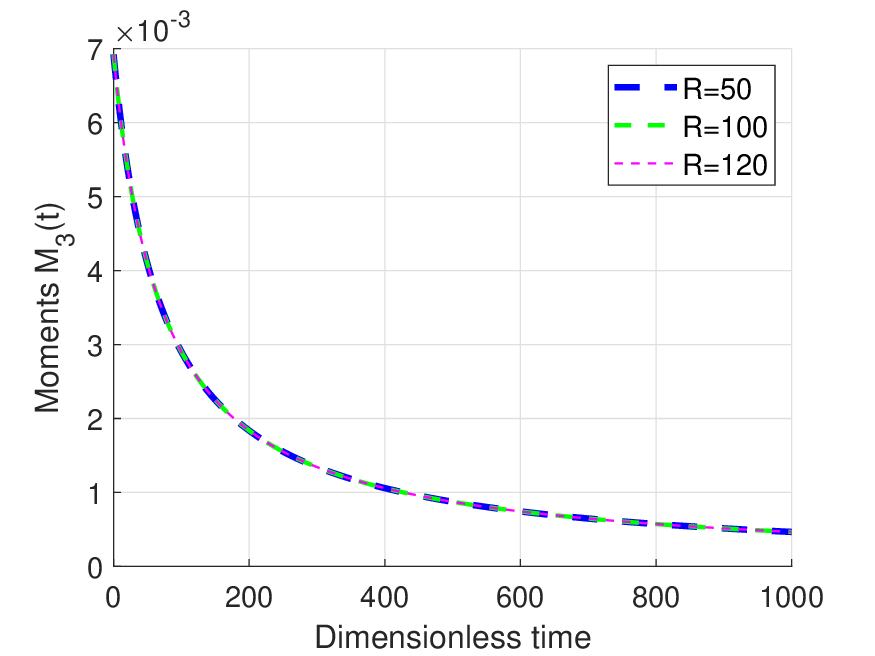}
		\caption{}
		\label{f4.3}
	\end{subfigure}
	\caption{Time evolution of the (a) first, (b) second, and (c) third moments for different truncation parameters $R$ with the same step length $h$ and degree of homogeneity $\theta$.}
	\label{f4}
\end{figure}

Now, we plot the first three moment functions for different truncation parameters while keeping the degree of the collision kernel and the number of grid points fixed. This implies that, while changing the truncation parameter, the length of the sub-intervals of the mesh is also changed. We perform this experiment using $100$ grid points. From Figures \ref{f5}, we observe that, since the length of the sub-intervals changes, the initial evolution also changes. However, after a period of time, the solutions converge to the same pattern throughout the evolution period. This phenomenon indicates that the scheme is quite stable with respect to the step length $h$.	
\begin{figure}[htp]
	\begin{subfigure}{.3\textwidth}
		\centering
		\includegraphics[width=1.0\textwidth]{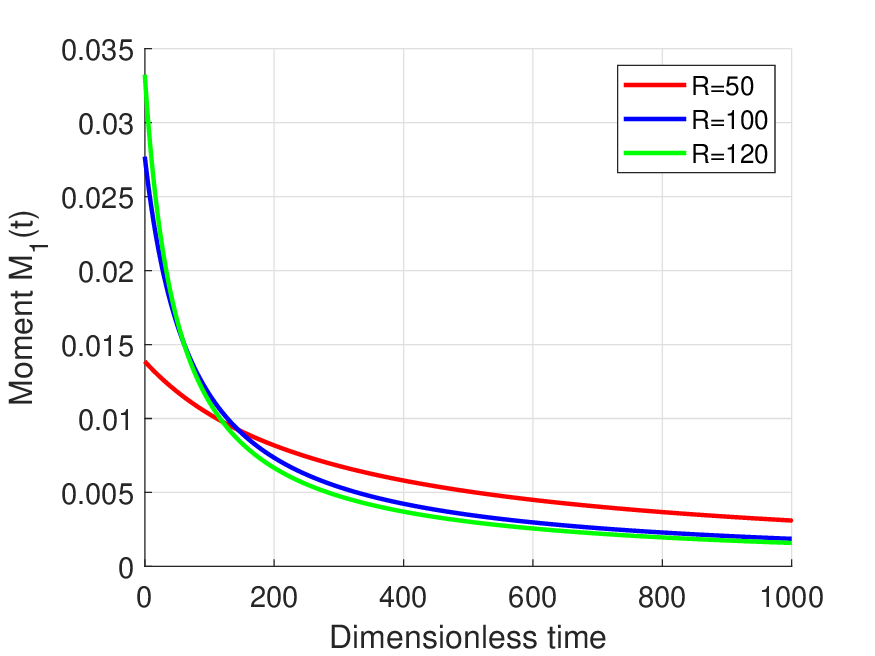}
		\caption{}
		\label{f5.1}
	\end{subfigure}
	\begin{subfigure}{.3\textwidth}
		\centering
		\includegraphics[width=1.0\textwidth]{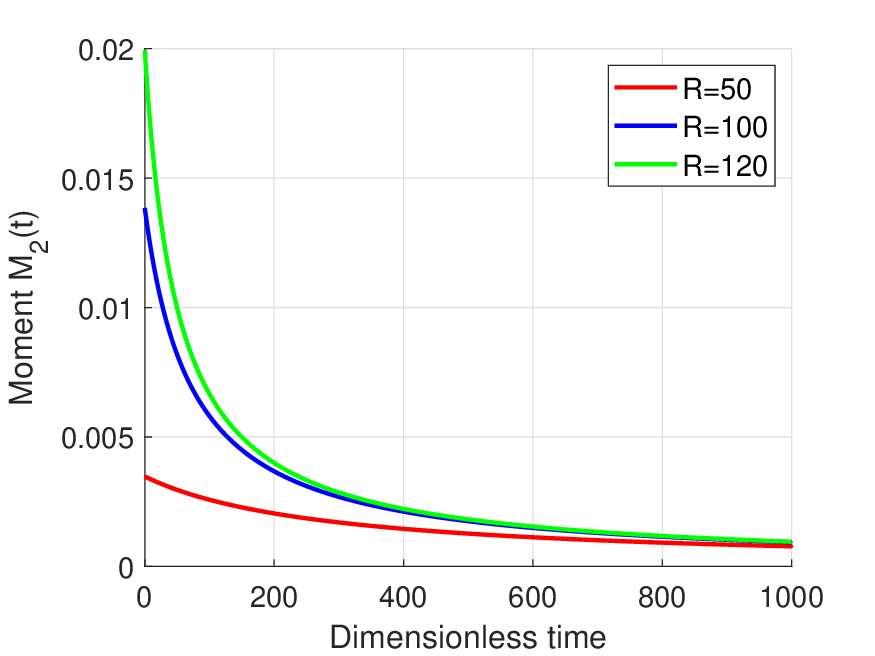}
		\caption{}
		\label{f5.2}
	\end{subfigure}
	\begin{subfigure}{.3\textwidth}
		\centering
		\includegraphics[width=1.0\textwidth]{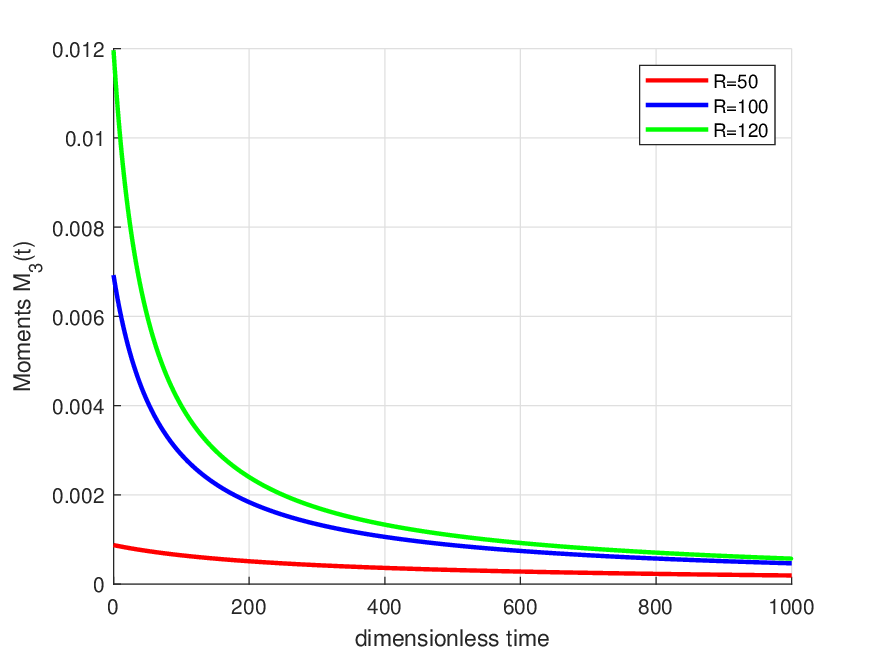}
		\caption{}
		\label{f5.3}
	\end{subfigure}
	\caption{Time evolution of the (a) first, (b) second, and (c) third moments for different truncation parameters $R$ and step lengths $h$ with the same degree of homogeneity $\theta$.}
	\label{f5}
\end{figure}

Now, we test the performance of our scheme for different collision kernels in the coagulation and fragmentation parts by choosing different degrees of homogeneity. In this regard, we choose the collision kernels as $K_1(\omega,\mu)= (\omega\mu)^\theta, K_2(\omega,\mu)= (\omega\mu)^\gamma,$ and $K_3(\omega,\mu) = (\omega\mu)^\delta$. We plot the time evolution of three different moments for various values of $\theta,\gamma$ and $\delta$ in Figure \ref{f_5}. 
\begin{figure}[htp]
	\begin{subfigure}{.3\textwidth}
		\centering
		\includegraphics[width=1.0\textwidth]{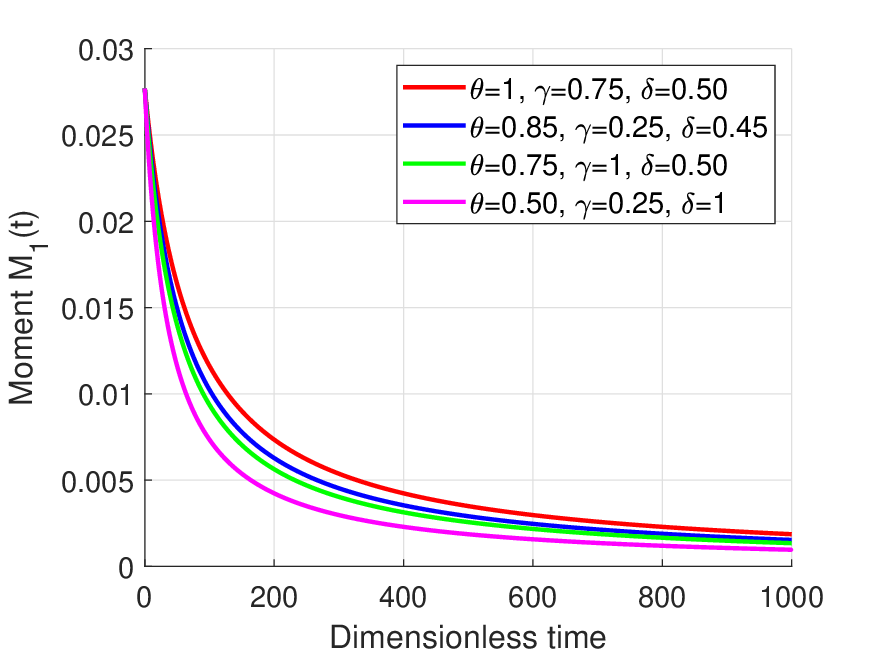}
		\caption{}
		\label{f_5.1}
	\end{subfigure}
	\begin{subfigure}{.3\textwidth}
		\centering
		\includegraphics[width=1.0\textwidth]{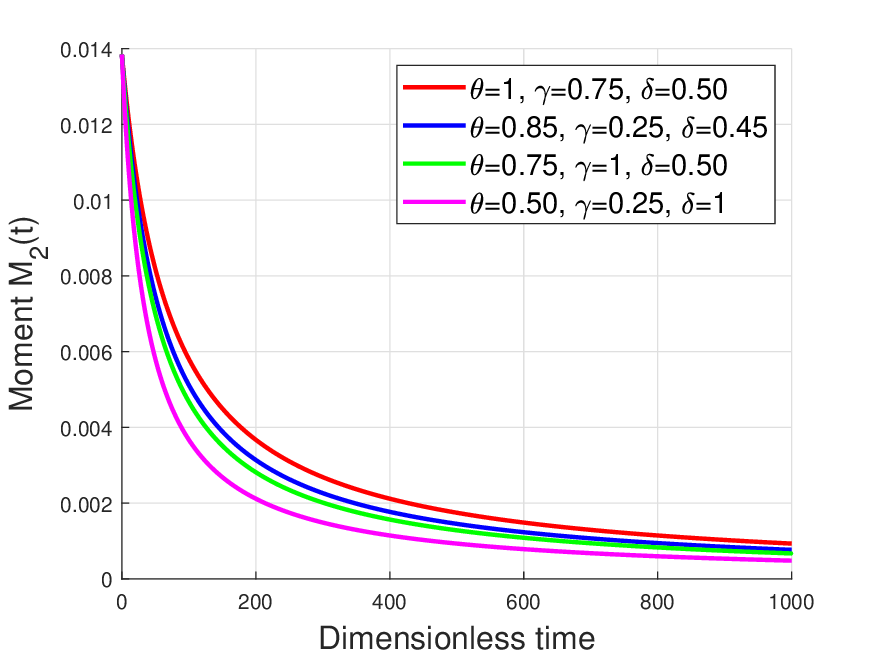}
		\caption{}
		\label{f_5.2}
	\end{subfigure}
	\begin{subfigure}{.3\textwidth}
		\centering
		\includegraphics[width=1.0\textwidth]{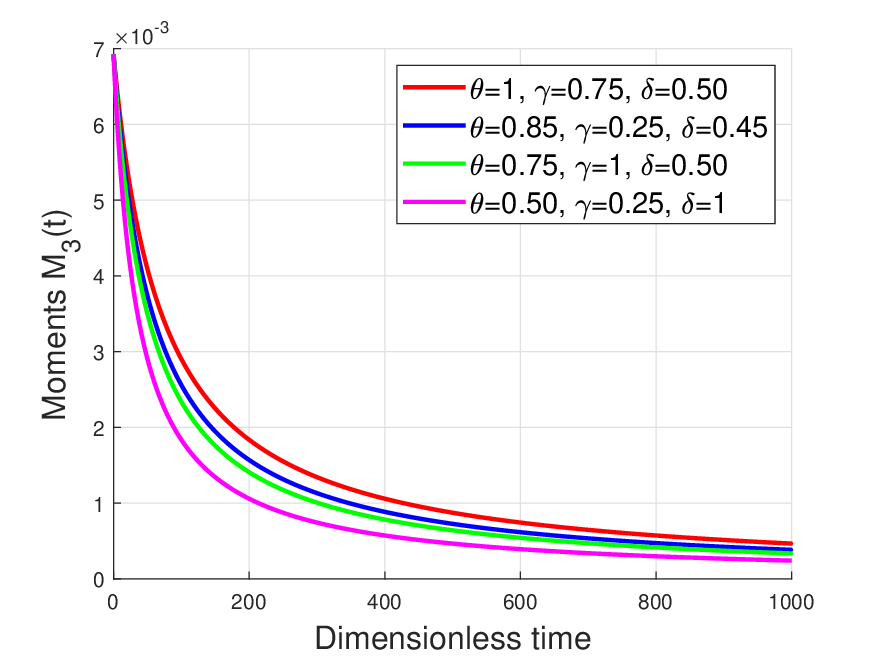}
		\caption{}
		\label{f_5.3}
	\end{subfigure}
	\caption{Time evolution of the (a) first, (b) second, and (c) third moments for different values of $\theta, \gamma$ and $\delta$ with the same truncation parameter $R$.}
	\label{f_5}
\end{figure}
To conserve the total energy, we have implemented our proposed weighted FVS \eqref{eq_7} for the same initial data \eqref{2.1} and set $\theta=\gamma=\delta=0.15,0.10, 0.05$. As discussed in the theoretical work \cite{soffer2020energy}, in this case, the energy cascade phenomenon \eqref{1.15} does not happen. Therefore, we expect to see a conservation of the energy of the numerical solutions. To observe the accuracy of the scheme \eqref{eq_7}, we plot the three moment functions $M_1(t)$, $M_2(t)$ and $M_3(t)$ for different degrees of homogeneity  with respect to the dimensionless time $t$. From Figure \ref{f7_1}, we can see that the first-order moment remains constant throughout the time domain. This implies that our proposed scheme \eqref{eq_7} successfully conserves energy for a quite large time domain.  From Figures \ref{f7_2} and \ref{f7_3}, we can observe that the second- and third-order moment functions $M_2(t)$, $M_3(t)$ show a tendency to stabilize after a short period of time.
\begin{figure}[htp]
	\begin{subfigure}{.3\textwidth}
		\centering
		\includegraphics[width=1.0\textwidth]{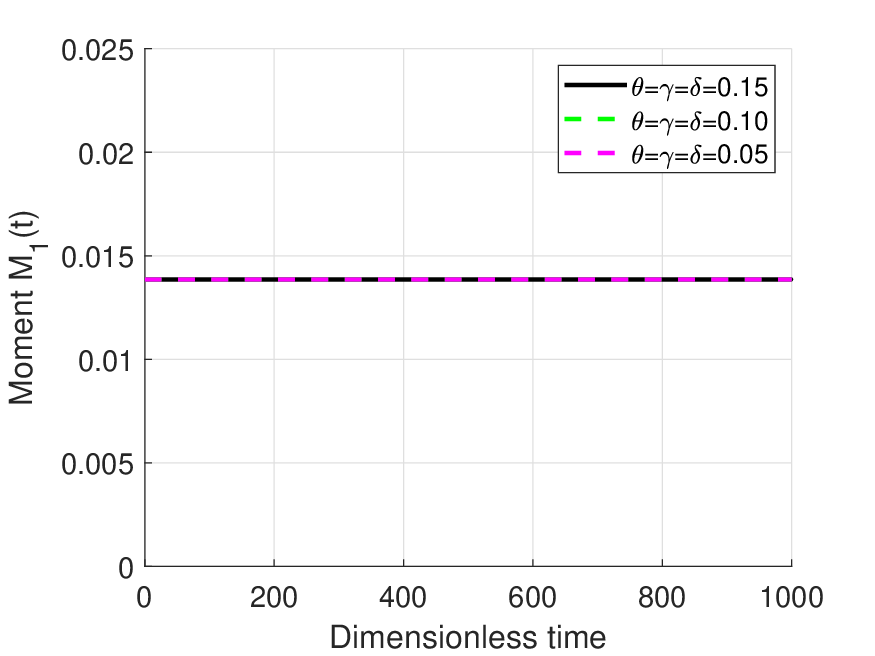}
		\caption{}
		\label{f7_1}
	\end{subfigure}
	\begin{subfigure}{.3\textwidth}
		\centering
		\includegraphics[width=1.0\textwidth]{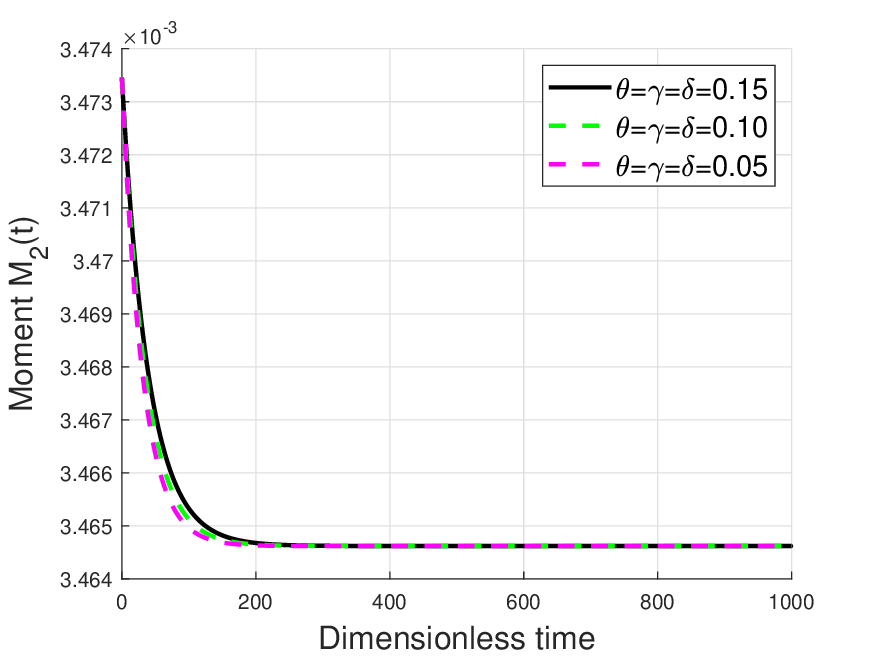}
		\caption{}
		\label{f7_2}
	\end{subfigure}
	\begin{subfigure}{.3\textwidth}
		\centering
		\includegraphics[width=1.0\textwidth]{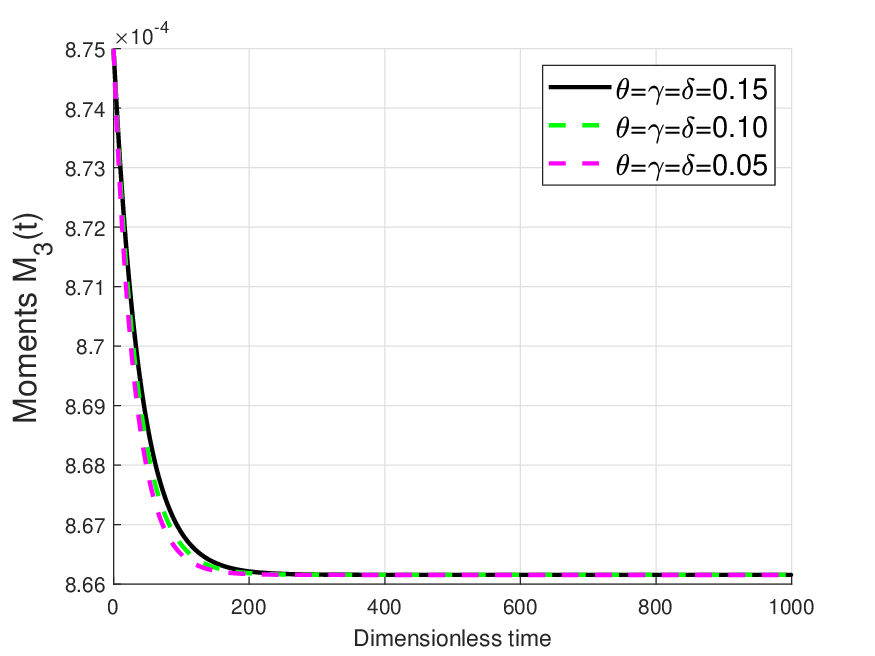}
		\caption{}
		\label{f7_3}
	\end{subfigure}
	\caption{Time evolution of the (a) first, (b) second, and (c) third moments for different values of $\theta$ with the same truncation parameter $R$.}
	\label{f7}
\end{figure}

\subsection{Test case II:}\label{test_3} The second numerical test employs an initial condition that is compactly supported at lower frequencies and is given by
\begin{align}\label{2.4}
	N_\omega^{in} = \left\{\begin{array}{ll}
		\exp\left(\frac{5}{\left|\omega-5\right|^2-1}\right), &\mbox{if}\quad\left|\omega-5\right|\le 1 \vspace{0.2cm}\\
		0 , &\mbox{if}\quad\left|\omega-5\right|>1.
	\end{array}\right.
\end{align}
To compute the solution, we consider the step length of the mesh points to be $h=0.5$ with the computation domain $[0,70]$. The initial condition and final state can be seen in Figure \ref{f13} for the collision kernels \eqref{2.3}. In Figures \ref{f13_1} and \ref{f13_2}, we observe that the energy is pushed slightly toward the origin at some $T_s \in \left[0,T\right)$ to $\omega= 0.2512$, away from its initial concentration at $\omega=4.7738$ at $t=0$.
\begin{figure}[htp]
	\begin{center}
		\begin{subfigure}{.3\textwidth}
			\centering
			\includegraphics[width=1.0\textwidth]{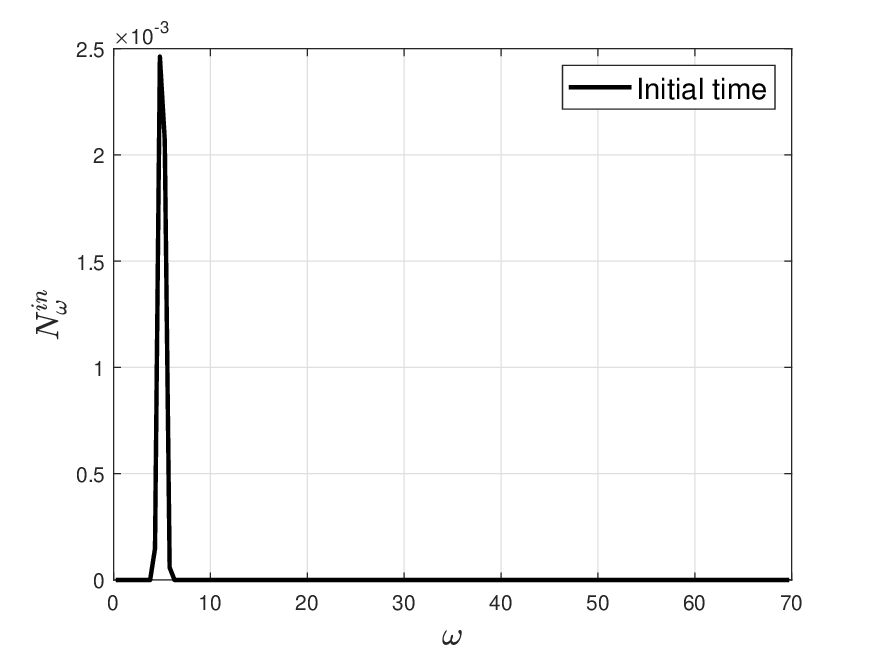}
			\caption{Initial Condition}
			\label{f13_1}
		\end{subfigure}
		\begin{subfigure}{.3\textwidth}
			\centering
			\includegraphics[width=1.0\textwidth]{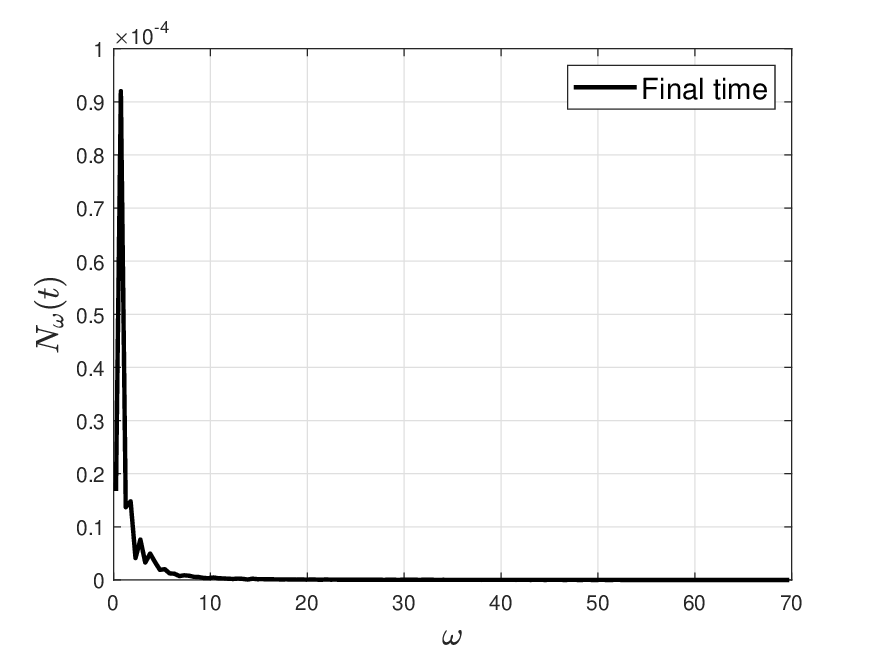}
			\caption{Final Condition}
			\label{f13_2}
		\end{subfigure}
		\caption{Evolution of wave density  $N_\omega$ using the FVS \eqref{eq_6} at the initial and final times.}
		\label{f13}
	\end{center}
\end{figure}

In Figure \ref{f14}, we have plotted the first three moment functions while allowing the degree of homogeneity of the collision kernels \eqref{2.3} to vary, while holding the truncation parameter fixed. We observe the decaying behavior of the total energy, as well as the other moments, similar to that in the previous test case.
\begin{figure}[htp]
	\begin{subfigure}{.3\textwidth}
		\centering
		\includegraphics[width=1.0\textwidth]{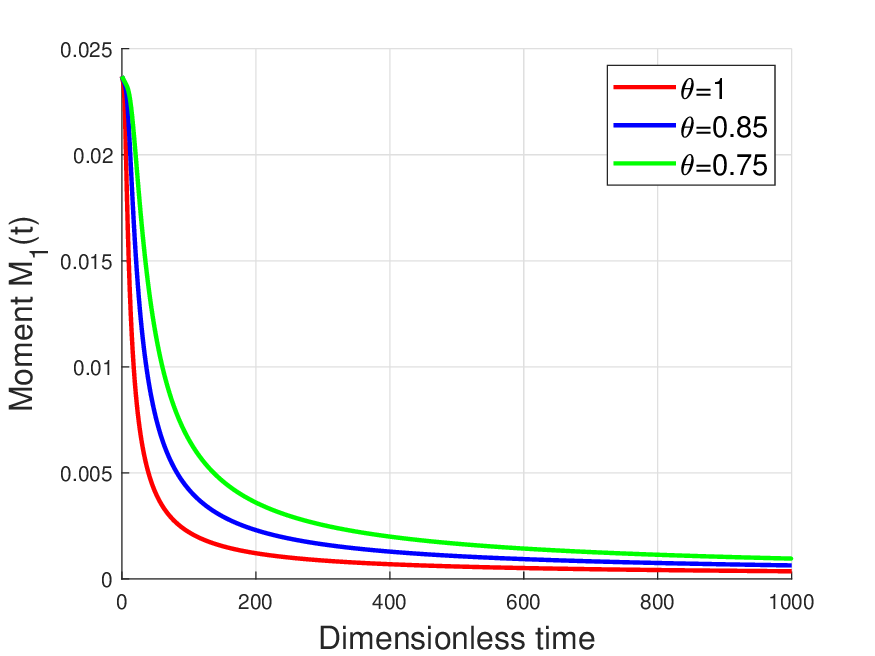}
		\caption{}
		\label{f14_1}
	\end{subfigure}
	\begin{subfigure}{.3\textwidth}
		\centering
		\includegraphics[width=1.0\textwidth]{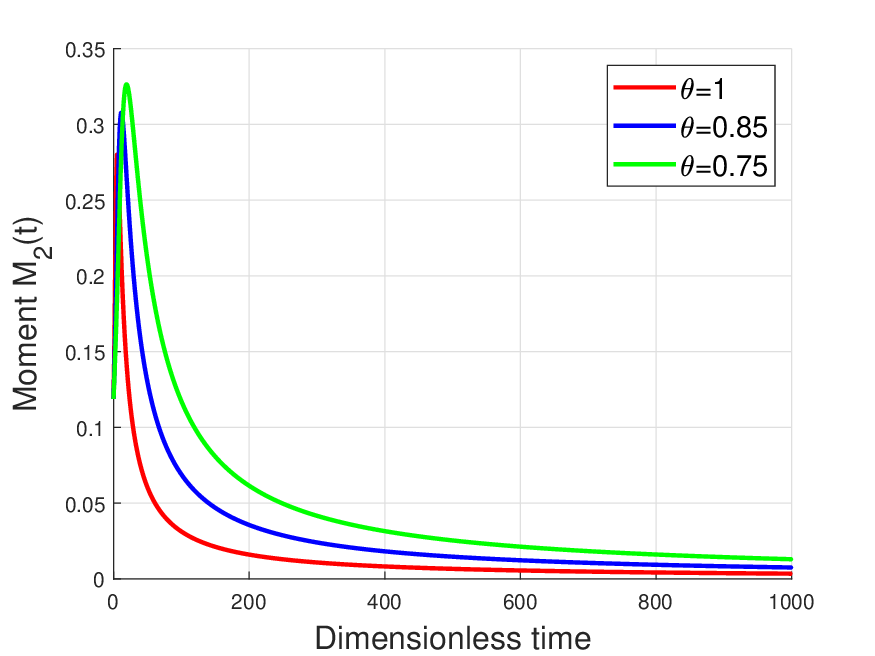}
		\caption{}
		\label{f14_2}
	\end{subfigure}
	\begin{subfigure}{.3\textwidth}
		\centering
		\includegraphics[width=1.0\textwidth]{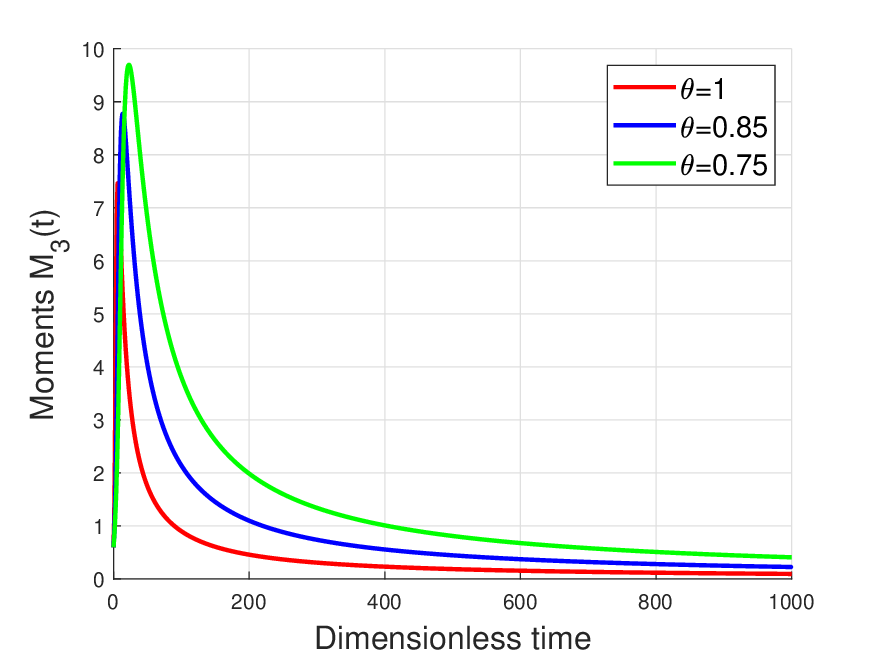}
		\caption{}
		\label{f14_3}
	\end{subfigure}
	\caption{Time evolution of the (a) first, (b) second, and (c) third moments for different values of $\theta$ with the same truncation parameter $R$.}
	\label{f14}
\end{figure}

In Figure \ref{f15}, the theoretical decay rate is compared with the decay of total energy for all considered values of $\theta$. The plot indicates that the total energy in the finite interval considered is conserved for a short time, as given in \cite{soffer2020energy}. After this time, the total energy in the interval begins to decay rapidly. As in the previous test cases, the numerical results show good agreement with the theoretical findings of \cite{soffer2020energy}. The figure suggests that the decay is bounded by $\mathcal{O}\left(\frac{1}{\sqrt{t}}\right)$, as shown in \eqref{1.15}.
\begin{figure}[htp]
	\begin{center}
		\begin{subfigure}{.3\textwidth}
			\centering
			\includegraphics[width=1.0\textwidth]{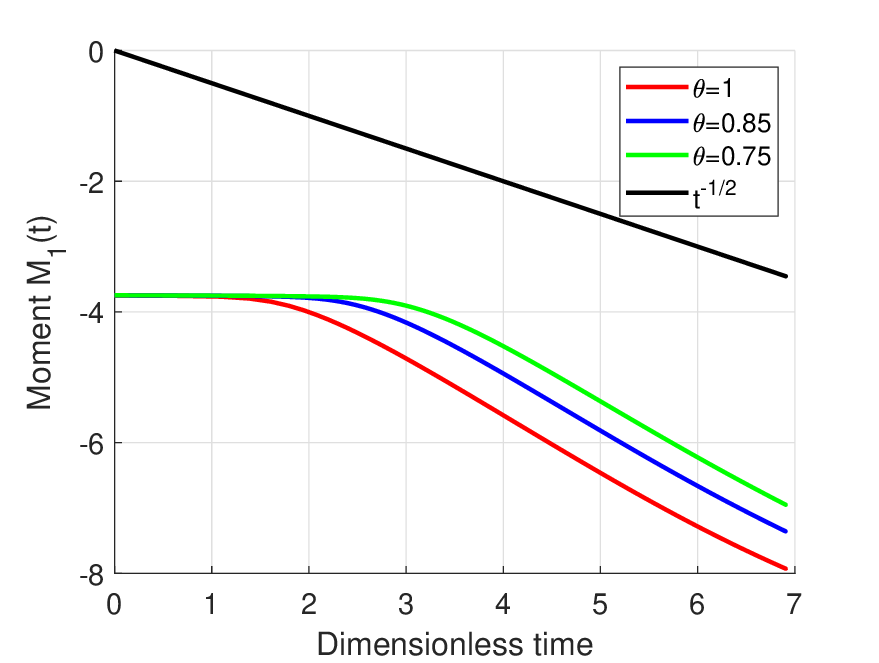}
			\caption{ }
			\label{f15_1}
		\end{subfigure}
		\caption{{Decay of total energy for different values of $\theta$ on a logarithmic scale.}}
		\label{f15}
	\end{center}
\end{figure}

We also perform the evolution of the moment functions in Figure \ref{f16}, when $\theta=1$ is fixed and the truncation parameter $R$ is varying.  In contrast to Figure \ref{f4} in the previous test case, the differences in moments are more distinguishable. However, as already mentioned, this is consistent with the previous analysis in \cite{soffer2020energy}.
\begin{figure}[htp]
	\begin{subfigure}{.3\textwidth}
		\centering
		\includegraphics[width=1.0\textwidth]{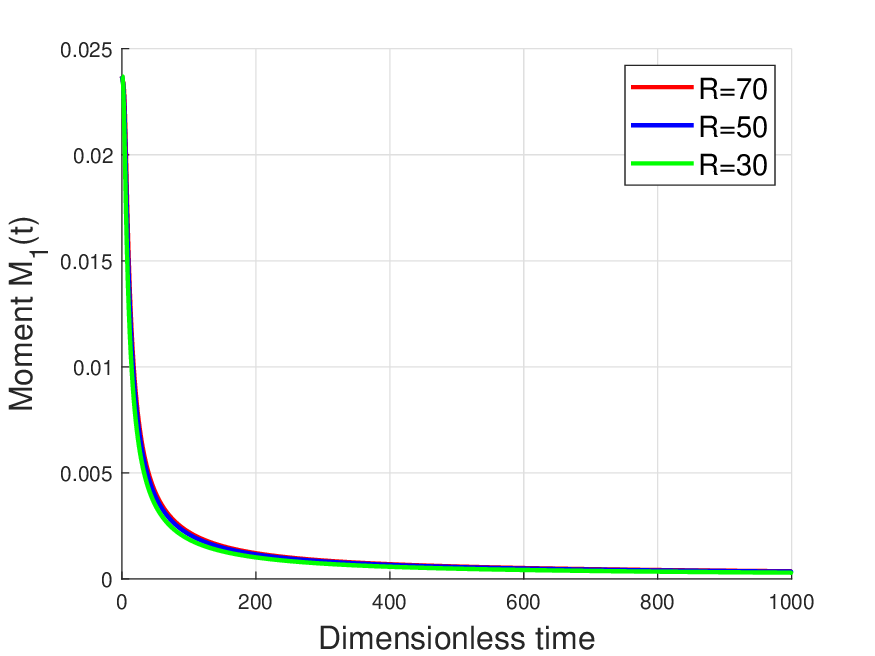}
		\caption{}
		\label{f16_1}
	\end{subfigure}
	\begin{subfigure}{.3\textwidth}
		\centering
		\includegraphics[width=1.0\textwidth]{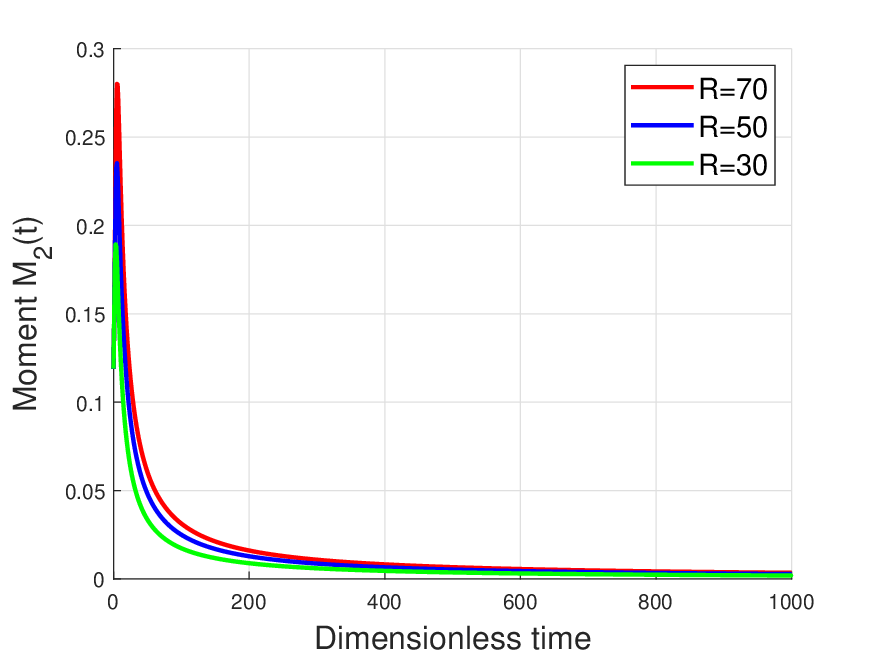}
		\caption{}
		\label{f16_2}
	\end{subfigure}
	\begin{subfigure}{.3\textwidth}
		\centering
		\includegraphics[width=1.0\textwidth]{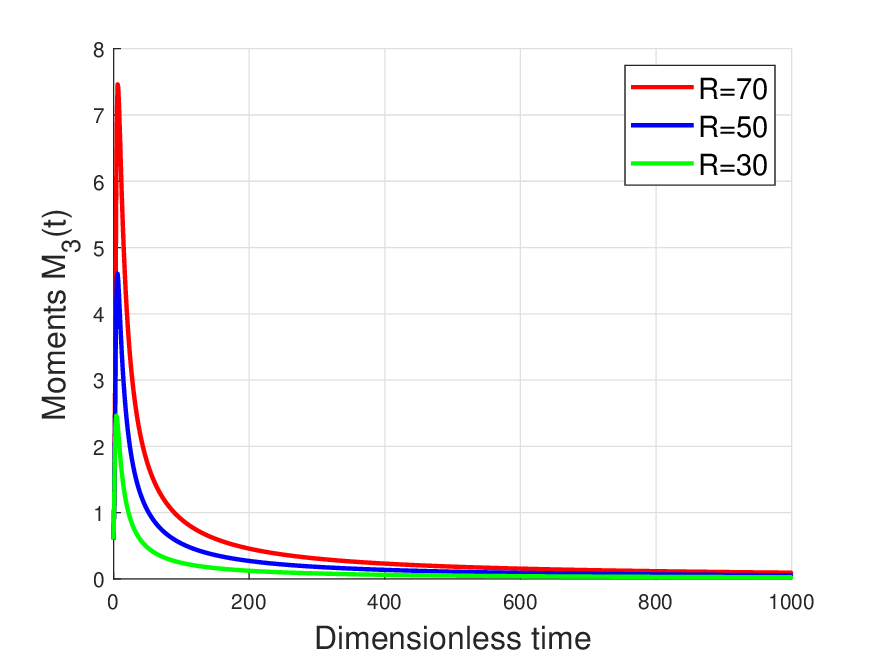}
		\caption{}
		\label{f16_3}
	\end{subfigure}
	\caption{Time evolution of the (a) first, (b) second, and (c) third moments for different truncation parameter $R$ with the same degree of homogeneity $\theta$.}
	\label{f16}
\end{figure}

\subsection{Experimental order of convergence}\label{EOC}
To validate the theoretical result on the order of convergence proved in the previous section, we conduct several numerical test problems to compute the \emph{Experimental Order of Convergence} (EOC). Specifically, we calculate the EOC for the first test case (\ref{test_1}) and second test case (\ref{test_3}), both involving wave interaction kernels with a homogeneity degree of one. Since the exact solution for the  fully nonlinear coagulation-fragmentation model equation \eqref{1.1} is not available in the literature for either test case, we compute the EOC numerically using the following formula:

\begin{align*}
	EOC = \ln\left(\frac{\|N_I-N_{2I}\|}{\|N_{2I}-N_{4I}\|}\right)/\ln(2).
\end{align*}
Here $\ds N_I$ denotes the approximate solution obtained by the FVS on a mesh with $I$ number of grid points. We perform both test cases on uniform and nonuniform grids, using a geometric grid for the nonuniform case. The geometric grid is derived from the smooth transformation $\ds \omega = \exp(\xi)$, where $\xi$ is a variable with a uniform mesh in {the computational domain $\left[\ln(1e-8), \ln(2)\right]$ and $\left[\ln(1e-8), \ln(10)\right]$ for the first and second test case}, respectively. Additionally, the simulations are run till time $t=200$ and $t = 50$ with time step $\Delta t =0.1$, for the Test case \ref{test_1} and Test case \ref{test_3}, respectively.

The table \ref{t1} provides information of the $L^1$ error and EOC for the initial data considered in the Test case \ref{test_1}.  From this we can observe that our proposed FVS has first order convergence on uniform as well as nonuniform grid. Which is a good agreement with the theoretical proof of convergence analysis stated in Theorem \ref{thm_3}. 

\begin{table}[!htp]
	\centering
	\begin{tabular}{ |c| c|c| c |c|}
		\hline
		& \multicolumn{1}{c}{ \hspace{1cm}Uniform grids} & & \multicolumn{1}{c}{\hspace{1cm}Nonuniform grids}& \\
		\hline
		Grid points & $L^1$ error  & EOC & $L^1$ error  & EOC  \\
		\hline
		\multirow {5}{2em}{}
		\hspace{-1cm} 60 & 9.72E-4 & 0      & 6.14E-3 & 0       \\
		120 & 4.54E-4 & 1.0972 & 2.61E-3 & 1.2333  \\
		240 & 2.37E-4 & 0.9366 & 1.18E-3 & 1.1401  \\
		480 & 1.22E-4 & 0.9571 & 3.63E-4 & 1.0038  \\
		\hline
	\end{tabular}
	\caption{Error and EOC for the Test case \ref{test_1}.}\label{t1}
	\begin{tabular}{ |c| c|c| c |c|}
		\hline
		& \multicolumn{1}{c}{ \hspace{1cm}Uniform grids} & & \multicolumn{1}{c}{\hspace{1cm}Nonuniform grids}& \\
		\hline
		Grid points & $L^1$ error  & EOC & $L^1$ error  & EOC  \\
		\hline
		\multirow {5}{2em}{}
		\hspace{-1cm} 60 & 3.44E-5 & 0   & 4.96E-4 & 0   \\
		120 & 1.95E-5 & 0.8205 & 4.95E-4 & 0.1130  \\
		240 & 1.04E-5 & 0.9107 & 3.09E-4 & 0.6784  \\
		480 & 5.36E-6 & 0.9478 & 1.61E-4 & 0.9377  \\
		\hline
	\end{tabular}
	\caption{Error and EOC for the Test case \ref{test_3}.}\label{t2}
\end{table}
In the second test case, we compute the $L^1$ error and EOC for the initial data considered in Test Case \ref{test_3}. The errors and EOC values for both uniform and nonuniform meshes are presented in Table \ref{t2}. From the results, it can be observed that, similar to the first test case, the second test case also demonstrates first-order convergence of the FVS. This further verifies the theoretical result under favorable conditions.

\section{Conclusions and further discussion}
This article presented a novel finite volume scheme designed to solve a fully nonlinear coagulation-fragmentation model derived from the $3$-wave kinetic equation. Notably, unlike the methods described in \cite{connaughton2010dynamical}, the proposed scheme is the first numerical approach capable of capturing the long-term asymptotic behavior of solutions to the $3$-wave turbulence kinetic equation without imposing additional restrictions.

Through numerical tests, the scheme successfully validates the energy cascade phenomenon for various degrees of homogeneity in the collision kernels. These results demonstrate strong agreement with the theoretical findings in \cite{soffer2020energy}. Figures \ref{f3} and \ref{f15} further confirm the theoretical predictions regarding the energy decay rate, as described by:
\begin{align*}
	\begin{split}
		\int_{0}^{R}\omega N_\omega(t) \dd \omega = \int_{\mathbb{R}_+}\chi_{[0,R]}(\omega) \omega N_\omega \dd \omega \le \mathcal{O}\left(\frac{1}{\sqrt{t}}\right),\quad \text{ as } t \longrightarrow\infty,
	\end{split}
\end{align*}
for any truncation parameter $R$. Additionally, Figure \ref{f_5} illustrates the scheme's ability to describe the energy cascade phenomenon across different degrees of homogeneity for the collision kernels. To address energy conservation for varying kernel homogeneity, a weighted finite volume scheme was also introduced. Figure \ref{f7_1} confirms the numerical validation of energy conservation for varying degree of homogeneity of collision kernels. To provide a solid theoretical foundation for our proposed finite volume scheme, we conducted a detailed convergence analysis, demonstrating first-order consistency. In addition to this theoretical validation, the first-order convergence was confirmed through the computation of the experimental order of convergence, as presented in Tables \ref{t1} and \ref{t2}.


\enlargethispage{20pt}

\section{Acknowledgment}We would like to express our gratitude to Dr. Iulia Cristian, Dr. Steven Walton, Professor Avy Soffer and Professor Enrique Zuazua for fruitful discussions on the topics. A. D. expresses gratitude to the Chair for Dynamics, Control, Machine Learning, and Numerics in the Department of Mathematics at FAU Erlangen-Nurnberg for their hospitality, where the work was conducted. M.-B. T is  funded in part by  a   Humboldt Fellowship,   NSF CAREER  DMS-2303146, and NSF Grants DMS-2204795, DMS-2305523,  DMS-2306379.


\vskip2pc


%
%
%


\bibliographystyle{elsarticle-num} 

\bibliography{Wave} 
\end{document}